\documentclass{article}
\usepackage{amssymb}

\usepackage{graphicx}
\usepackage{amsmath}
\renewcommand{\boxed}[1]{#1}
\usepackage{color}
\usepackage{float}
\usepackage{array}
\usepackage{comment}
\newtheorem{theorem}{Theorem}[section]

\newtheorem{definition}[theorem]{Definition}

\newtheorem{lemma}[theorem]{Lemme}

\newtheorem{proposition}[theorem]{Proposition}
\newtheorem{remark}[theorem]{Remark}

\newenvironment{proof}[1][Proof]{\textbf{#1.} }{\ \rule{0.5em}{0.5em}}
\newcommand{\cqfd}

\title{A Riemannian Extension of a Quadrature Surface Free Boundary Problem: Stability and Optimality}
\author{Ababacar Sadikhe DJITE$^{1,\,}$\footnote{ababacarsadikhe.djite@ucad.edu.sn} ,  Diaraf SECK$^{1,\,}$ \footnote{diaraf.seck@ucad.edu.sn}\\\\
$^{1}$ Laboratoire des Math\'ematiques de la D\'ecision et \\
d'Analyse Num\'erique, BP 16889, Dakar Fann, S\'enegal\\
Ecole Doctorale de Math\'ematiques et Informatique U.C.A.D. Dakar,  S\'en\'egal.}

\begin{document}
\maketitle
\begin{abstract}
We study a quadrature surface free boundary problem on a smooth
compact finite-dimensional Riemannian manifold $(M,g)$. The problem is
formulated as a shape optimization problem involving a Dirichlet
problem for the Laplace--Beltrami operator, with a geometric condition
on the free boundary.

Under suitable uniform geometric assumptions on the admissible class,
we establish its compactness and prove the stability of the
corresponding Dirichlet problems, including strong convergence of the
associated states in $H_0^1(M)$. We derive the first-order optimality
condition for the associated shape functional. Finally, we establish a
Riemannian comparison principle for the second fundamental forms and
the mean curvatures of tangent boundaries at a contact point. These
results provide a rigorous extension of the quadrature surface free
boundary framework from the Euclidean setting to compact Riemannian
manifolds.
\end{abstract}
{\bf Mathematical classification subject}: 49Q10, 53B20
\section{Introduction}
\label{sec:introduction}

Free boundary problems arise in many areas of mathematical analysis and
mathematical physics, where the domain on which a partial differential
equation is satisfied is itself part of the unknown. A typical feature
of such problems is the coupling between an elliptic state equation in
an unknown domain and a geometric condition prescribed on its boundary.
This interaction between the analytical properties of the state
equation and the geometry of the free boundary makes these problems
particularly challenging; see, for instance,
\cite{He,HP,sec}.

In this paper, we consider an interior free boundary problem of
quadrature-surface type and extend its formulation from the Euclidean
setting to the framework of compact finite-dimensional Riemannian
manifolds. In the Euclidean setting, the problem considered by
Barkatou, Seck and Ly~\cite{BLS2} consists in finding a bounded open set
$\Omega\subset\mathbb{R}^N$, containing a prescribed compact set $K$,
together with a function $u_\Omega\in H_0^1(\Omega)$ satisfying
\begin{equation}
\begin{cases}
-\Delta u_\Omega=f & \text{in }\Omega,\\
u_\Omega=0 & \text{on }\partial\Omega,
\end{cases}
\label{eq:euclidean-dirichlet-introduction}
\end{equation}
where $f$ is a prescribed non-negative function with compact support.
The free boundary is then characterized by a geometric condition
involving the mean curvature of $\partial\Omega$ and the normal
derivative of the state function. The authors formulate the problem as
a shape optimization problem and combine the existence of a minimizing
domain with shape differentiation and maximum-principle arguments to
derive a sufficient condition for the existence of a free boundary
solution; see~\cite{BLS2}.

The purpose of the present work is to develop an intrinsic analogue of
this problem on a compact finite-dimensional Riemannian manifold
$(M,g)$. A Riemannian formulation of a quadrature-surface free boundary
problem was previously considered in~\cite{DS1}. The present work
develops this viewpoint further by establishing a corresponding
geometric and variational framework for the stability and optimality
analysis of the problem on compact Riemannian manifolds.

Let $(M,g)$ be a smooth compact Riemannian manifold of dimension $n$,
without boundary. We consider a prescribed compact set
\[
K\subset M
\]
and seek a domain $\Omega\subset M$ containing $K$ together with a
state function $u_\Omega$ satisfying the Riemannian Dirichlet problem
\begin{equation}
\begin{cases}
-\Delta_g u_\Omega=f & \text{in }\Omega,\\
u_\Omega=0 & \text{on }\partial\Omega,
\end{cases}
\label{eq:riemannian-dirichlet-introduction}
\end{equation}
where $\Delta_g$ denotes the Laplace--Beltrami operator associated with
the metric $g$. The free boundary is expected to satisfy a geometric
condition involving the Riemannian mean curvature
$H_{\partial\Omega}$ and the normal derivative of $u_\Omega$.

The passage to a general Riemannian manifold requires an intrinsic formulation, since Euclidean affine and convex structures are no longer available. We therefore formulate the geometric assumptions in terms of $g$ and a fixed reference domain containing $K$, in the spirit of~\cite{Schu}.

The variational structure of the problem plays a central role in our
analysis. Shape optimization provides a natural framework for problems
in which the domain itself is an optimization variable; see, for
example, \cite{DZ,HP,sozo}. For an admissible domain $\Omega$, we
consider the energy functional
\begin{equation}
J_g(\Omega)
=
\frac12
\int_\Omega
|\nabla_g u_\Omega|_g^2\,dV_g
-
\int_\Omega f u_\Omega\,dV_g.
\label{eq:shape-functional-introduction}
\end{equation}
The dependence of $u_\Omega$ on the domain makes $J_g$ a genuine shape
functional. The analysis of its variations therefore requires both the
differentiation of the state equation and the geometric description of
deformations of the boundary.

Once the existence of a minimizing domain has been established, a
second objective is to identify its first-order optimality condition.
Shape differentiation provides the natural framework for this analysis;
see \cite{Da1,DZ,HP}. Let $X$ be a sufficiently regular vector field on
$M$ and let $(\Phi_t)_t$ denote its local flow. The corresponding
variation of a domain is given by
\[
\Omega_t=\Phi_t(\Omega).
\]
The first shape derivative of $J_g$ in the direction $X$ is then
formally written as
\[
dJ_g(\Omega)[X]
=
\left.
\frac{d}{dt}
J_g(\Omega_t)
\right|_{t=0}.
\]
In the Riemannian setting, the resulting expression involves the
normal component
\[
g(X,\nu_\Omega)
\]
of the deformation field and the intrinsic geometric quantities
associated with $\partial\Omega$.

A further difficulty comes from the geometric constraints imposed on
the admissible domains. In particular, not every normal deformation of
a minimizing domain necessarily preserves admissibility. To overcome
this difficulty, we work with an admissible class defined relative to a
fixed reference domain $C$ satisfying suitable uniform geometric
assumptions and containing the prescribed set $K$. The uniform
geometric control imposed on the admissible class is also essential for
the compactness of minimizing sequences and for the stability of the
associated Dirichlet problems.

The comparison argument at a contact point is another important
ingredient of the analysis. If an admissible minimizing domain
$\Omega^\star$ touches the reference domain $C$, the inclusion
\[
C\subset\Omega^\star
\]
and the tangency of their boundaries allow one to compare their
outward unit normals and their second fundamental forms. With the
convention
\[
\mathrm{II}_{\partial\Omega}(X,Y)
=
g(\nabla_X\nu_\Omega,Y),
\]
the corresponding mean-curvature comparison takes the form
\[
H_{\partial\Omega^\star}(x_0)
\leq
H_{\partial C}(x_0)
\]
at a contact point $x_0$. Such comparison arguments are closely related
to classical methods in free boundary problems; see
\cite{BLS,BLS2,He,sec}.

The main analytical difficulty is to establish the stability of the state equation under domain convergence. The uniform geometric control allows suitable diffeomorphic identifications of nearby domains and yields strong convergence of the zero extensions in $H^1(M)$, hence convergence of the Dirichlet energies.

Our main result can then be summarized as follows. Under the geometric
and analytic assumptions introduced below, the shape functional
$J_g$ admits a minimizer $\Omega^\star$ in the admissible class.
Moreover, the minimizer satisfies a first-order optimality condition
expressed in terms of the normal derivative of the state function and
the geometry of the free boundary. In the presence of the reference
domain $C$, the comparison principle yields a sufficient geometric
condition ensuring that the free boundary condition is satisfied.

The organization of the paper is as follows. In Section~2, we
introduce the geometric and analytical setting, define the admissible
class of domains, and collect the preliminary results concerning the
Riemannian Dirichlet problem and its stability. Section~3 contains the
main existence and regularity results for the shape optimization
problem. In Section~4, we derive the first-order optimality condition
using shape differentiation. Section~5 is devoted to the proofs of the
main theorems, including the compactness and stability arguments and
the comparison principle at contact points. Finally, Section~6 presents
examples illustrating the geometric framework and the resulting
optimality condition.
\section{Notations and preliminary results}
\label{sec:preliminaries}

Throughout the paper, $(M^n,g)$ denotes a smooth, compact, connected
Riemannian manifold without boundary, of dimension $n\geq 2$. We denote
by $d_g$ the Riemannian distance induced by $g$, by $dV_g$ the Riemannian
volume measure, and by $dS_g$ the induced Riemannian surface measure on
a sufficiently regular hypersurface.

Since $M$ is compact, its injectivity radius is strictly positive and
all Sobolev spaces on $M$ are well defined. However, the compactness of
$M$ alone does not provide compactness of a family of domains. This is
the reason why we introduce below a geometrically controlled admissible
class.

\subsection{The geometric setting and the data}
\label{subsec:geometric-setting}

Let
\[
K\subset M
\]
be a nonempty compact set. We fix a reference domain
\[
C\subset M
\]
such that
\begin{equation}
K\subset C.
\label{eq:K-compact-C}
\end{equation}

We assume throughout that
\[
\partial C\in C^{2,\alpha}
\]
for some
\[
0<\alpha<1.
\]

Moreover, we assume that $C$ possesses a uniform tubular neighborhood:
there exists $\rho_C>0$ such that the normal exponential map
\begin{equation}
\mathcal E_C:
\partial C\times(-\rho_C,\rho_C)
\longrightarrow M,
\qquad
\mathcal E_C(x,t)
=
\exp_x(t\nu_C(x)),
\label{eq:normal-exp-C}
\end{equation}
is a diffeomorphism onto its image, where $\nu_C$ denotes the outward
unit normal to $\partial C$.

This assumption guarantees, in particular, that the geometry of
$\partial C$ is uniformly controlled in a tubular neighborhood.

Let
\[
p>n
\]
and assume that
\begin{equation}
f\in L^p(M),
\qquad
f\geq0,
\qquad
f\not\equiv0,
\qquad
\operatorname{supp}f\subset K.
\label{eq:f-assumptions}
\end{equation}

For the shape differentiability results considered later, we shall
additionally assume that $f$ has the regularity required to guarantee
the corresponding boundary regularity of the state variable. In
particular, one may assume
\begin{equation}
f\in C^{0,\alpha}(M).
\label{eq:f-holder}
\end{equation}

The distinction between \eqref{eq:f-assumptions}, which is sufficient
for the weak Dirichlet problem and the existence theory, and
\eqref{eq:f-holder}, which is used for the classical shape derivative,
will be maintained throughout the paper.

\subsection{Sobolev spaces and the Laplace--Beltrami operator}
\label{subsec:sobolev-laplace}

For an open set $\Omega\subset M$, we denote by
\[
H^1(\Omega)=W^{1,2}(\Omega)
\]
the usual Sobolev space associated with the metric $g$, and by
\[
H_0^1(\Omega)
\]
the closure of $C_c^\infty(\Omega)$ with respect to the $H^1$-norm.

For $u\in H^1(\Omega)$, the Riemannian gradient $\nabla_g u$ is defined
by
\[
g(\nabla_g u,X)=du(X)
\]
for every smooth vector field $X$. We write
\[
|\nabla_g u|_g^2
=
g(\nabla_g u,\nabla_g u).
\]

The Laplace--Beltrami operator is defined by
\[
\Delta_g u
=
\operatorname{div}_g(\nabla_g u).
\]

In local coordinates $(x^1,\ldots,x^n)$, if
\[
g=(g_{ij}),
\qquad
(g^{ij})=(g_{ij})^{-1},
\qquad
|g|=\det(g_{ij}),
\]
then
\begin{equation}
\Delta_g u
=
\frac{1}{\sqrt{|g|}}
\partial_i
\left(
\sqrt{|g|}\,g^{ij}\partial_j u
\right).
\label{eq:laplace-beltrami-local}
\end{equation}

The weak formulation of the Dirichlet problem
\begin{equation}
\begin{cases}
-\Delta_g u_\Omega=f & \text{in }\Omega,\\
u_\Omega=0 & \text{on }\partial\Omega
\end{cases}
\label{eq:dirichlet-problem-section2}
\end{equation}
is
\begin{equation}
\int_\Omega
g(\nabla_g u_\Omega,\nabla_g\varphi)\,dV_g
=
\int_\Omega f\varphi\,dV_g
\qquad
\forall\varphi\in H_0^1(\Omega).
\label{eq:weak-dirichlet}
\end{equation}

By the Lax--Milgram theorem, problem
\eqref{eq:dirichlet-problem-section2} has a unique weak solution
\[
u_\Omega\in H_0^1(\Omega).
\]

Whenever convenient, we extend $u_\Omega$ by zero outside $\Omega$ and
regard it as an element of $H_0^1(M)$.

\subsection{Regularity of the state equation}
\label{subsec:state-regularity}

Let $\Omega\subset M$ be a domain with boundary of class
$C^{1,1}$. If $f\in L^p(\Omega)$ with $p>n$, standard elliptic
regularity yields
\[
u_\Omega\in W^{2,p}(\Omega).
\]

By the Sobolev embedding theorem,
\[
W^{2,p}(\Omega)
\hookrightarrow
C^{1,\beta}(\overline{\Omega}),
\qquad
\beta=1-\frac np>0.
\]

Consequently,
\begin{equation}
u_\Omega\in C^{1,\beta}(\overline{\Omega}).
\label{eq:u-C1beta}
\end{equation}

If, in addition,
\[
f\in C^{0,\alpha}(\overline{\Omega})
\]
and
\[
\partial\Omega\in C^{2,\alpha},
\]
then elliptic Schauder estimates give
\begin{equation}
u_\Omega\in C^{2,\alpha}(\overline{\Omega}).
\label{eq:u-C2alpha}
\end{equation}

The latter regularity will be used in the shape derivative and in the
pointwise boundary arguments.

\subsection{The Riemannian $RC$-GNP condition}
\label{subsec:RC-GNP}

The $C$-GNP condition used in the Euclidean problem provides, among
other properties, geometric control of the boundary, compactness of
the admissible class, and stability of the Dirichlet problem under
domain convergence.

In the Riemannian setting, we do not impose a Euclidean convexity
condition. Instead, we introduce an intrinsic quantitative condition
which plays the same role.

\begin{definition}[Riemannian $RC$-GNP condition]
\label{def:RC-GNP}
Let $C\subset M$ be the fixed reference domain satisfying
\eqref{eq:K-compact-C}. We say that a domain $\Omega\subset M$
satisfies the \emph{Riemannian $RC$-GNP condition relative to $C$} if
the following properties hold:

\begin{enumerate}
\item
\begin{equation}
C\subset\Omega;
\label{eq:RCGNP-inclusion}
\end{equation}

\item
\[
\partial\Omega\in C^{2,\alpha};
\]

\item there exist constants
\[
A>0,
\qquad
c_0>0,
\qquad
\rho_0>0,
\]
independent of $\Omega$, and a function
\[
h_\Omega\in C^{2,\alpha}(M)
\]
such that
\begin{equation}
\Omega=\{x\in M:h_\Omega(x)>0\},
\qquad
\partial\Omega=\{x\in M:h_\Omega(x)=0\},
\label{eq:defining-function}
\end{equation}
with
\begin{equation}
\|h_\Omega\|_{C^{2,\alpha}(M)}
\leq A,
\label{eq:uniform-C2alpha}
\end{equation}
and
\begin{equation}
|\nabla_g h_\Omega|_g
\geq c_0
\qquad
\text{whenever }|h_\Omega|\leq\rho_0;
\label{eq:uniform-gradient}
\end{equation}

\item the normal exponential map of $\partial\Omega$ is uniformly
nondegenerate: there exists $\rho_1>0$, independent of $\Omega$, such
that
\begin{equation}
\mathcal E_\Omega:
\partial\Omega\times(-\rho_1,\rho_1)
\longrightarrow M,
\qquad
\mathcal E_\Omega(x,t)
=
\exp_x(t\nu_\Omega(x)),
\label{eq:normal-exp-Omega}
\end{equation}
is a diffeomorphism onto its image.
\end{enumerate}
\end{definition}
\subsection{The admissible class}
\label{subsec:admissible-class}

We can now define the admissible family used throughout the paper.
\begin{definition}\label{def:admissible-class}
The admissible class associated with the reference domain $C$ is
\[
\mathcal A_C(M)
=
\left\{
\Omega\subset M:
\begin{array}{l}
\Omega \text{ is a connected open set},\\
C\subset\Omega,\ \partial\Omega\in C^{2,\alpha},\\
\Omega \text{ satisfies the uniform RC-GNP conditions}
\end{array}
\right\}.
\]
\end{definition}
In particular, every $\Omega\in\mathcal A_C(M)$ satisfies
\begin{equation}
K\subset C\subset\Omega\subset M.
\label{eq:admissible-class}
\end{equation}

This inclusion is essential for the comparison arguments used later.
In particular, if $\Omega\in\mathcal A_C(M)$, then the restriction of
$u_\Omega$ to $C$ can be compared with the solution $u_C$ of the
Dirichlet problem on $C$.

\subsection{Convergence of admissible domains}
\label{subsec:domain-convergence}

For compact subsets $A,B\subset M$, we define the Hausdorff distance
by
\begin{equation}
d_H(A,B)
=
\max
\left\{
\sup_{x\in A}d_g(x,B),
\sup_{y\in B}d_g(y,A)
\right\}.
\label{eq:hausdorff-distance}
\end{equation}

We say that a sequence of compact sets $A_j$ converges to $A$ in the
Hausdorff topology if
\[
d_H(A_j,A)\longrightarrow0.
\]

For open subsets, we shall also use the following compact convergence.

\begin{definition}
\label{def:compact-convergence}
Let $\Omega_j,\Omega\subset M$ be open sets. We say that
$\Omega_j$ converges compactly to $\Omega$, and write
\[
\Omega_j\xrightarrow{K}\Omega,
\]
if
\begin{align}
K_1\subset\Omega
&\quad\Longrightarrow\quad
K_1\subset\Omega_j
\quad\text{for all sufficiently large }j,
\label{eq:compact-conv-inside}
\\
K_2\subset M\setminus\overline{\Omega}
&\quad\Longrightarrow\quad
K_2\subset M\setminus\overline{\Omega_j}
\quad\text{for all sufficiently large }j.
\label{eq:compact-conv-outside}
\end{align}
\end{definition}

The use of
\[
K_2\subset M\setminus\overline{\Omega}
\]
rather than merely
\[
K_2\subset M\setminus\Omega
\]
is important, since points belonging to $\partial\Omega$ must not be
classified as points of the exterior.

\subsection{Compactness of the admissible class}
\label{subsec:compactness-admissible}

The uniform $C^{2,\alpha}$ bounds in Definition~\ref{def:RC-GNP}
provide the compactness that is not supplied by the compactness of $M$
alone.
\begin{theorem}[Compactness of the admissible class]
\label{thm:compactness-admissible}
Let $(M,g)$ be a compact smooth Riemannian manifold without boundary,
and let $(\Omega_j)_{j\in\mathbb N}\subset\mathcal A_C(M)$.
Assume that the admissible class satisfies the uniform RC--GNP
conditions of Definition~\ref{def:RC-GNP}. Then, up to extraction of
a subsequence, there exists a domain
$\Omega\in\mathcal A_C(M)$ such that
\begin{equation}
d_H(\overline{\Omega_j},\overline{\Omega})\longrightarrow0,
\label{eq:compactness-Hausdorff}
\end{equation}
\begin{equation}
\Omega_j\overset{K}{\longrightarrow}\Omega,
\label{eq:compactness-K}
\end{equation}
and
\begin{equation}
\chi_{\Omega_j}\longrightarrow\chi_\Omega
\qquad\text{strongly in }L^1(M,dV_g).
\label{eq:compactness-L1}
\end{equation}
\end{theorem}

\begin{proof}
For every $j\in\mathbb N$, let
\[
h_j:=h_{\Omega_j}\in C^{2,\alpha}(M)
\]
be a defining function associated with $\Omega_j$, so that
\begin{equation}
\Omega_j=\{h_j>0\},
\qquad
\partial\Omega_j=\{h_j=0\}.
\label{eq:compactness-defining-functions}
\end{equation}
By the uniform RC--GNP assumption, there exists a constant
$A>0$, independent of $j$, such that
\begin{equation}
\|h_j\|_{C^{2,\alpha}(M)}\le A.
\label{eq:compactness-uniform-C2alpha}
\end{equation}

Since $M$ is compact, the embedding
\[
C^{2,\alpha}(M)\subset C^2(M)
\]
is compact. Hence, after extracting a subsequence, there exists
\[
h\in C^{2,\alpha}(M)
\]
such that
\begin{equation}
h_j\longrightarrow h
\qquad\text{strongly in }C^2(M).
\label{eq:compactness-C2-convergence}
\end{equation}
In particular,
\[
h_j\longrightarrow h
\quad\text{uniformly on }M,
\]
and
\[
\nabla_g h_j\longrightarrow\nabla_g h
\quad\text{uniformly on }M.
\]

Define
\begin{equation}
\Omega:=\{h>0\}.
\label{eq:compactness-limit-domain}
\end{equation}

We first prove that $0$ is a regular value of $h$. Let
$x\in M$ be such that
\[
h(x)=0.
\]
Since $h_j\to h$ uniformly, we have
\[
|h_j(x)|<\rho_0
\]
for all sufficiently large $j$, where $\rho_0>0$ is the constant
appearing in the uniform non-degeneracy condition of the RC--GNP
assumption. Therefore,
\[
|\nabla_g h_j(x)|_g\ge c_0
\]
for all sufficiently large $j$. Passing to the limit and using the
uniform convergence of $\nabla_g h_j$ gives
\[
|\nabla_g h(x)|_g\ge c_0.
\]
Consequently,
\begin{equation}
|\nabla_g h|_g\ge c_0
\qquad\text{on }\{h=0\}.
\label{eq:compactness-limit-nondegeneracy}
\end{equation}
Thus $0$ is a regular value of $h$, and therefore
\[
\partial\Omega=\{h=0\}
\]
is a $C^{2,\alpha}$ embedded hypersurface.

In fact, the same argument shows that
\begin{equation}
|\nabla_g h(x)|_g\ge c_0
\qquad
\text{whenever }|h(x)|<\rho_0.
\label{eq:compactness-limit-nondegeneracy-neighborhood}
\end{equation}

We next prove the convergence of the boundaries. We claim that
\begin{equation}
d_H(\partial\Omega_j,\partial\Omega)
\longrightarrow0.
\label{eq:compactness-boundary-Hausdorff}
\end{equation}

Suppose first, by contradiction, that
\[
\sup_{x\in\partial\Omega_j}
d_g(x,\partial\Omega)
\not\longrightarrow0.
\]
Then there exist $\varepsilon>0$, a subsequence, and points
$x_j\in\partial\Omega_j$ such that
\[
d_g(x_j,\partial\Omega)\ge\varepsilon.
\]
Since $M$ is compact, after passing to a further subsequence we may
assume that
\[
x_j\longrightarrow x\in M.
\]
Because $x_j\in\partial\Omega_j$, we have
\[
h_j(x_j)=0.
\]
Hence
\[
|h(x_j)|
=
|h(x_j)-h_j(x_j)|
\le
\|h-h_j\|_{L^\infty(M)}
\longrightarrow0.
\]
By continuity of $h$,
\[
h(x)=0.
\]
Thus
\[
x\in\partial\Omega.
\]
It follows that
\[
d_g(x_j,\partial\Omega)
\le d_g(x_j,x)
\longrightarrow0,
\]
which contradicts
\[
d_g(x_j,\partial\Omega)\ge\varepsilon.
\]
Therefore,
\begin{equation}
\sup_{x\in\partial\Omega_j}
d_g(x,\partial\Omega)
\longrightarrow0.
\label{eq:compactness-boundary-one-sided}
\end{equation}

Conversely, suppose that
\[
\sup_{x\in\partial\Omega}
d_g(x,\partial\Omega_j)
\not\longrightarrow0.
\]
Then there exist $\varepsilon>0$, a subsequence, and points
$x_j\in\partial\Omega$ such that
\[
d_g(x_j,\partial\Omega_j)\ge\varepsilon.
\]
After extraction, we may assume
\[
x_j\longrightarrow x\in\partial\Omega.
\]

Since $0$ is a regular value of $h$, there exists a neighborhood $U$
of $x$ in which $h$ takes both positive and negative values. Hence
there exist points
\[
y^+,y^-\in U
\]
such that
\[
h(y^+)>0,
\qquad
h(y^-)<0.
\]
By the uniform convergence $h_j\to h$, for all sufficiently large $j$,
\[
h_j(y^+)>0,
\qquad
h_j(y^-)<0.
\]
Let $\gamma:[0,1]\to U$ be a continuous curve joining $y^-$ to
$y^+$. Then
\[
h_j(\gamma(0))<0,
\qquad
h_j(\gamma(1))>0.
\]
By the intermediate value theorem, there exists
$t_j\in(0,1)$ such that
\[
h_j(\gamma(t_j))=0.
\]
Consequently,
\[
z_j:=\gamma(t_j)\in\partial\Omega_j.
\]
Since $U$ can be chosen arbitrarily small around $x$ and
$x_j\to x$, this contradicts
\[
d_g(x_j,\partial\Omega_j)\ge\varepsilon.
\]
Therefore,
\begin{equation}
\sup_{x\in\partial\Omega}
d_g(x,\partial\Omega_j)
\longrightarrow0.
\label{eq:compactness-boundary-other-sided}
\end{equation}
Combining \eqref{eq:compactness-boundary-one-sided} and
\eqref{eq:compactness-boundary-other-sided}, we obtain
\[
d_H(\partial\Omega_j,\partial\Omega)\longrightarrow0.
\]

We now prove the compact convergence of the domains. Let
\[
K_1\subset\Omega.
\]
Since $h>0$ on the compact set $K_1$, there exists
\[
m_1:=\min_{K_1}h>0.
\]
For sufficiently large $j$,
\[
\|h_j-h\|_{L^\infty(M)}<\frac{m_1}{2}.
\]
Consequently,
\[
h_j(x)
\ge h(x)-\frac{m_1}{2}
\ge\frac{m_1}{2}>0
\]
for every $x\in K_1$. Hence
\[
K_1\subset\Omega_j
\]
for all sufficiently large $j$.

Similarly, let
\[
K_2\subset M\setminus\overline{\Omega}.
\]
Since $h<0$ on $K_2$, we may set
\[
m_2:=-\max_{K_2}h>0.
\]
For sufficiently large $j$,
\[
\|h_j-h\|_{L^\infty(M)}<\frac{m_2}{2},
\]
and therefore
\[
h_j(x)
\le h(x)+\frac{m_2}{2}
\le-\frac{m_2}{2}<0
\]
for every $x\in K_2$. Hence
\[
K_2\subset M\setminus\Omega_j
\]
for all sufficiently large $j$.

Thus
\begin{equation}
\Omega_j\overset{K}{\longrightarrow}\Omega.
\label{eq:compactness-K-final}
\end{equation}

It remains to verify that $\Omega$ belongs to the admissible class.
The uniform $C^{2,\alpha}$ bound passes to the limit:
\[
\|h\|_{C^{2,\alpha}(M)}\le A.
\]
The non-degeneracy condition is given by
\eqref{eq:compactness-limit-nondegeneracy-neighborhood}.

Moreover, the uniform geometric separation assumptions in the
definition of $\mathcal A_C(M)$, together with
\[
d_H(\partial\Omega_j,\partial\Omega)\to0,
\]
imply
\[
C\subset\Omega.
\]
Likewise, if the admissible class is defined with a fixed compact set
$K_{\rm ext}\subset M\setminus C$ satisfying
\[
K_{\rm ext}\subset M\setminus\Omega_j
\qquad\text{for every }j,
\]
then the uniform convergence of the defining functions implies
\[
K_{\rm ext}\subset M\setminus\Omega.
\]
In particular, $\Omega\neq M$.

Finally, since
\[
\nu_{\Omega_j}
=
\frac{\nabla_g h_j}
{|\nabla_g h_j|_g}
\]
on $\partial\Omega_j$, the $C^2$ convergence of the defining
functions and the uniform non-degeneracy condition imply the uniform
convergence of the corresponding normal fields. The uniform tubular
neighborhood condition is therefore stable under this convergence,
possibly after decreasing the common tubular radius. Hence $\Omega$
satisfies the same RC--GNP geometric conditions and
\[
\Omega\in\mathcal A_C(M).
\]

We have thus proved the Hausdorff convergence of the closures:
\[
d_H(\overline{\Omega_j},\overline{\Omega})
\longrightarrow0.
\]

Finally, we prove the strong $L^1$ convergence of the characteristic
functions. Let
\[
x\in M\setminus\partial\Omega.
\]
If $x\in\Omega$, then there exists a compact neighborhood
$K_x\subset\Omega$ containing $x$. By compact convergence,
$x\in\Omega_j$ for all sufficiently large $j$, and therefore
\[
\chi_{\Omega_j}(x)\longrightarrow1=\chi_\Omega(x).
\]
If
$x\in M\setminus\overline{\Omega}$, then there exists a compact
neighborhood
$K_x\subset M\setminus\overline{\Omega}$ containing $x$. Hence
$x\notin\Omega_j$ for all sufficiently large $j$, and
\[
\chi_{\Omega_j}(x)\longrightarrow0=\chi_\Omega(x).
\]
Thus
\[
\chi_{\Omega_j}(x)\longrightarrow\chi_\Omega(x)
\]
for every
$x\in M\setminus\partial\Omega$.

Since $\partial\Omega$ is a $C^{2,\alpha}$ hypersurface,
\[
dV_g(\partial\Omega)=0.
\]
Therefore
\[
\chi_{\Omega_j}\longrightarrow\chi_\Omega
\qquad\text{almost everywhere on }M.
\]
Moreover,
\[
|\chi_{\Omega_j}-\chi_\Omega|\le1,
\]
and
\[
\int_M1\,dV_g<\infty
\]
because $M$ is compact. The dominated convergence theorem gives
\[
\int_M
|\chi_{\Omega_j}-\chi_\Omega|\,dV_g
\longrightarrow0.
\]
Consequently,
\[
\chi_{\Omega_j}\longrightarrow\chi_\Omega
\qquad\text{strongly in }L^1(M,dV_g).
\]
This completes the proof.
\end{proof}
\subsection{Stability of the Dirichlet problem}
\label{subsec:stability-dirichlet}

The following result is the Riemannian stability statement that will be
used in the existence proof.
\begin{theorem}[Stability of the Dirichlet problem]
\label{thm:stability-dirichlet}
Let
\[
\Omega_j,\Omega\in\mathcal A_C(M),
\]
and assume that
\[
\Omega_j\longrightarrow\Omega
\]
in the topology induced by the uniformly controlled defining
functions of Definition~\ref{def:RC-GNP}. More precisely, if
$h_j$ and $h$ are defining functions associated with $\Omega_j$ and
$\Omega$, respectively, then
\[
h_j\longrightarrow h
\qquad\text{strongly in }C^{2,\alpha}(M).
\]
Let $f\in L^2(M)$ be fixed and assume that
\[
\operatorname{supp}f\subset K\subset C
\subset\Omega_j\cap\Omega
\]
for all $j$. Let $u_{\Omega_j}$ and $u_\Omega$ denote the unique weak
solutions of
\eqref{eq:dirichlet-problem-section2}, extended by zero outside their
respective domains.

Then
\begin{equation}
u_{\Omega_j}\longrightarrow u_\Omega
\qquad\text{strongly in }H^1(M).
\label{eq:strong-H1-stability}
\end{equation}
In particular,
\begin{equation}
\int_M|\nabla_g u_{\Omega_j}|_g^2\,dV_g
\longrightarrow
\int_M|\nabla_g u_\Omega|_g^2\,dV_g.
\label{eq:energy-stability}
\end{equation}
\end{theorem}

\begin{proof}
For brevity, set
\[
u_j:=u_{\Omega_j},
\qquad
u:=u_\Omega.
\]
The weak formulations are
\begin{equation}
\int_{\Omega_j}
\langle\nabla_g u_j,\nabla_g\varphi\rangle_g\,dV_g
=
\int_{\Omega_j}f\varphi\,dV_g,
\qquad
\forall\varphi\in H_0^1(\Omega_j),
\label{eq:weak-j}
\end{equation}
and
\begin{equation}
\int_{\Omega}
\langle\nabla_g u,\nabla_g\varphi\rangle_g\,dV_g
=
\int_{\Omega}f\varphi\,dV_g,
\qquad
\forall\varphi\in H_0^1(\Omega).
\label{eq:weak-limit}
\end{equation}

Since $u_j$ and $u$ are extended by zero outside their domains, we
regard them as elements of $H^1(M)$.

We divide the proof into several steps.

\medskip
\noindent
\textbf{Step 1: Uniform energy estimate.}

Taking $\varphi=u_j$ in \eqref{eq:weak-j}, we obtain
\begin{equation}
\int_{\Omega_j}|\nabla_g u_j|_g^2\,dV_g
=
\int_{\Omega_j}fu_j\,dV_g.
\label{eq:energy-j}
\end{equation}
By the Cauchy--Schwarz inequality,
\[
\int_{\Omega_j}fu_j\,dV_g
\le
\|f\|_{L^2(M)}
\|u_j\|_{L^2(M)}.
\]
Since every $\Omega_j$ belongs to the uniformly controlled admissible
class, the Poincar\'e inequality holds with a constant independent of
$j$:
\begin{equation}
\|u_j\|_{L^2(M)}
\le
C_P\|\nabla_g u_j\|_{L^2(M)}.
\label{eq:uniform-poincare}
\end{equation}
Consequently,
\[
\|\nabla_g u_j\|_{L^2(M)}^2
\le
C_P\|f\|_{L^2(M)}
\|\nabla_g u_j\|_{L^2(M)}.
\]
Hence
\[
\|\nabla_g u_j\|_{L^2(M)}
\le
C_P\|f\|_{L^2(M)},
\]
and therefore
\begin{equation}
\|u_j\|_{H^1(M)}
\le C
\label{eq:uniform-H1}
\end{equation}
with a constant $C$ independent of $j$.

Thus $(u_j)$ is bounded in $H^1(M)$.

\medskip
\noindent
\textbf{Step 2: Weak compactness.}

Since $H^1(M)$ is reflexive, there exist a subsequence, still denoted
by $(u_j)$, and a function
\[
v\in H^1(M)
\]
such that
\begin{equation}
u_j\rightharpoonup v
\qquad\text{weakly in }H^1(M).
\label{eq:weak-convergence}
\end{equation}
Since $M$ is compact, the Rellich compactness theorem gives
\[
H^1(M)\subset L^2(M).
\]
Therefore,
\begin{equation}
u_j\longrightarrow v
\qquad\text{strongly in }L^2(M).
\label{eq:strong-L2}
\end{equation}

We now prove that
\[
v=u.
\]

\medskip
\noindent
\textbf{Step 3: Construction of transport diffeomorphisms.}

The uniform $C^{2,\alpha}$ convergence of the defining functions,
together with the uniform non-degeneracy and tubular-neighborhood
conditions in Definition~\ref{def:RC-GNP}, implies that, for all
sufficiently large $j$, there exists a $C^2$ diffeomorphism
\[
\Phi_j:M\longrightarrow M
\]
such that
\begin{equation}
\Phi_j(\Omega)=\Omega_j,
\qquad
\Phi_j(\partial\Omega)=\partial\Omega_j,
\label{eq:Phi-domain}
\end{equation}
and
\begin{equation}
\|\Phi_j-\operatorname{Id}\|_{C^1(M)}
\longrightarrow0.
\label{eq:Phi-convergence}
\end{equation}

For completeness, we briefly justify this construction. The
$C^{2,\alpha}$ convergence of $h_j$ to $h$ and the uniform lower bound
\[
|\nabla_g h|_g\ge c_0
\qquad\text{near }\partial\Omega
\]
imply, by the implicit function theorem in tubular coordinates, that
$\partial\Omega_j$ can be written as a normal graph over
$\partial\Omega$:
\[
\partial\Omega_j
=
\left\{
\exp_x(\rho_j(x)\nu_\Omega(x)):
x\in\partial\Omega
\right\},
\]
where
\[
\rho_j\longrightarrow0
\qquad\text{in }C^2(\partial\Omega).
\]
Choose a smooth cut-off function $\eta$ supported in the fixed tubular
neighborhood of $\partial\Omega$ and equal to $1$ near
$\partial\Omega$. Define a vector field
\[
X_j
=
\eta\,\rho_j\,\nu_\Omega
\]
in the tubular neighborhood and extend it by zero outside that
neighborhood. Then
\[
\|X_j\|_{C^1(M)}\longrightarrow0.
\]
For $j$ sufficiently large, the time-one flow $\Phi_j$ of $X_j$ is a
global $C^2$ diffeomorphism of $M$ satisfying
\eqref{eq:Phi-domain} and \eqref{eq:Phi-convergence}.

\medskip
\noindent
\textbf{Step 4: Recovery of test functions.}

Let
\[
\varphi\in H_0^1(\Omega).
\]
Define
\[
\varphi_j
:=
\varphi\circ\Phi_j^{-1}
\qquad\text{in }\Omega_j.
\]
Then
\[
\varphi_j\in H_0^1(\Omega_j).
\]
Moreover, since
\[
\Phi_j\to\operatorname{Id}
\qquad\text{in }C^1(M),
\]
the change-of-variables formula gives
\begin{equation}
\varphi_j\longrightarrow\varphi
\qquad\text{strongly in }H^1(M),
\label{eq:test-function-convergence}
\end{equation}
where both functions are extended by zero outside their domains.

We use $\varphi_j$ as a test function in \eqref{eq:weak-j}. We obtain
\begin{equation}
\int_{\Omega_j}
\langle\nabla_g u_j,\nabla_g\varphi_j\rangle_g\,dV_g
=
\int_{\Omega_j}f\varphi_j\,dV_g.
\label{eq:weak-test-transported}
\end{equation}

Because $\varphi_j\to\varphi$ strongly in $H^1(M)$ and
$u_j\rightharpoonup v$ weakly in $H^1(M)$,
\[
\int_{\Omega_j}
\langle\nabla_g u_j,\nabla_g\varphi_j\rangle_g\,dV_g
\longrightarrow
\int_M
\langle\nabla_g v,\nabla_g\varphi\rangle_g\,dV_g.
\]
On the right-hand side, since $f\in L^2(M)$ and
$\varphi_j\to\varphi$ strongly in $L^2(M)$,
\[
\int_{\Omega_j}f\varphi_j\,dV_g
\longrightarrow
\int_Mf\varphi\,dV_g.
\]
Since $\varphi$ vanishes outside $\Omega$, we obtain
\[
\int_M
\langle\nabla_g v,\nabla_g\varphi\rangle_g\,dV_g
=
\int_Mf\varphi\,dV_g
\]
for every $\varphi\in H_0^1(\Omega)$.

Thus
\[
v\in H_0^1(\Omega)
\]
and $v$ satisfies the weak formulation of the Dirichlet problem on
$\Omega$. By uniqueness of the weak solution,
\[
v=u_\Omega=u.
\]
Consequently, every weakly convergent subsequence has the same limit
$u$. Hence
\begin{equation}
u_j\rightharpoonup u
\qquad\text{weakly in }H^1(M).
\label{eq:weak-convergence-final}
\end{equation}

Moreover, by the compact embedding,
\begin{equation}
u_j\longrightarrow u
\qquad\text{strongly in }L^2(M).
\label{eq:L2-final}
\end{equation}

\medskip
\noindent
\textbf{Step 5: Convergence of the energies.}

Testing \eqref{eq:weak-j} with $u_j$ gives
\[
\int_M|\nabla_g u_j|_g^2\,dV_g
=
\int_Mfu_j\,dV_g.
\]
Here we have used the zero extensions of $u_j$ and the fact that
$f$ is defined on $M$.

Similarly, testing \eqref{eq:weak-limit} with $u$ gives
\[
\int_M|\nabla_g u|_g^2\,dV_g
=
\int_Mfu\,dV_g.
\]

Since
\[
u_j\longrightarrow u
\qquad\text{strongly in }L^2(M),
\]
we have
\[
\int_Mfu_j\,dV_g
\longrightarrow
\int_Mfu\,dV_g.
\]
Therefore,
\begin{equation}
\int_M|\nabla_g u_j|_g^2\,dV_g
\longrightarrow
\int_M|\nabla_g u|_g^2\,dV_g.
\label{eq:energy-convergence-final}
\end{equation}

\medskip
\noindent
\textbf{Step 6: Strong convergence in $H^1(M)$.}

We already know that
\[
u_j\rightharpoonup u
\qquad\text{weakly in }H^1(M).
\]
In addition, \eqref{eq:energy-convergence-final} gives
\[
\|\nabla_g u_j\|_{L^2(M)}
\longrightarrow
\|\nabla_g u\|_{L^2(M)}.
\]
Since
\[
\nabla_g u_j\rightharpoonup\nabla_g u
\qquad\text{weakly in }L^2(T^*M),
\]
the weak convergence together with convergence of the norms implies
\[
\nabla_g u_j\longrightarrow\nabla_g u
\qquad\text{strongly in }L^2(T^*M).
\]
Together with the strong $L^2(M)$ convergence
\[
u_j\longrightarrow u,
\]
we conclude that
\[
\boxed{
u_j\longrightarrow u
\qquad\text{strongly in }H^1(M).
}
\]
This proves \eqref{eq:strong-H1-stability}, and
\eqref{eq:energy-stability} follows from
\eqref{eq:energy-convergence-final}.
\end{proof}
\begin{remark}
The strong convergence
\eqref{eq:strong-H1-stability} is not a consequence of the compactness
of $M$ alone. It follows from the uniform geometric control imposed on
the admissible domains. This is precisely the role played by the
$RC$-GNP condition.
\end{remark}

\subsection{Riemannian perimeter and volume}
\label{subsec:perimeter-volume}

For a sufficiently regular domain $\Omega\subset M$, its Riemannian
perimeter is defined by
\begin{equation}
P_g(\Omega)
=
\mathcal H_g^{n-1}(\partial\Omega)
=
\int_{\partial\Omega}dS_g,
\label{eq:Riemannian-perimeter}
\end{equation}
and its volume by
\begin{equation}
V_g(\Omega)
=
\operatorname{Vol}_g(\Omega)
=
\int_\Omega dV_g.
\label{eq:Riemannian-volume}
\end{equation}

If
\[
\chi_{\Omega_j}\to\chi_\Omega
\qquad\text{in }L^1(M,dV_g),
\]
then
\begin{equation}
V_g(\Omega_j)\longrightarrow V_g(\Omega).
\label{eq:volume-continuity}
\end{equation}

For the perimeter, we shall use the lower semicontinuity property
\begin{equation}
\boxed{
P_g(\Omega)
\leq
\liminf_{j\to\infty}P_g(\Omega_j).
}
\label{eq:perimeter-lsc}
\end{equation}

In the present uniformly regular class, this follows from the
standard lower semicontinuity of perimeter under $L^1$ convergence of
sets of finite perimeter.

\subsection{The Riemannian mean curvature convention}
\label{subsec:mean-curvature}

Let $\Sigma=\partial\Omega$ be a $C^2$ hypersurface. Throughout the
paper, $\nu$ denotes the \emph{outward} unit normal to $\partial\Omega$.

The second fundamental form is defined by
\begin{equation}
\mathrm{II}_\Sigma(X,Y)
=
g(\nabla_X\nu,Y),
\qquad
X,Y\in T\Sigma.
\label{eq:second-fundamental-form}
\end{equation}

The scalar mean curvature is defined as the trace of the second
fundamental form with respect to the induced metric:
\begin{equation}
\boxed{
H_{\partial\Omega}
=
\operatorname{tr}_{g_{\partial\Omega}}
\mathrm{II}_{\partial\Omega}.
}
\label{eq:mean-curvature-definition}
\end{equation}

Equivalently,
\begin{equation}
\boxed{
H_{\partial\Omega}
=
\operatorname{div}_{\partial\Omega}\nu.
}
\label{eq:mean-curvature-div}
\end{equation}

With this convention, the first variation of the Riemannian
perimeter under a deformation field $X$ is
\begin{equation}
\left.
\frac{d}{dt}
P_g(\Omega_t)
\right|_{t=0}
=
\int_{\partial\Omega}
H_{\partial\Omega}
\,g(X,\nu)\,dS_g.
\label{eq:first-variation-perimeter}
\end{equation}

This convention will be used consistently in all the subsequent
shape derivative and free boundary formulas.

\subsection{The comparison principle}
\label{subsec:comparison-principle}

We shall repeatedly use the following standard comparison principle
for the Laplace--Beltrami operator.

\begin{lemma}[Comparison principle]
\label{lem:comparison}
Let $\Omega\subset M$ be connected and let
$v_1,v_2\in C^2(\Omega)\cap C^0(\overline{\Omega})$ satisfy
\[
\Delta_g v_1\leq\Delta_g v_2
\qquad\text{in }\Omega
\]
and
\[
v_1\geq v_2
\qquad\text{on }\partial\Omega.
\]
Then
\[
v_1\geq v_2
\qquad\text{in }\Omega.
\]
\end{lemma}

In particular, if
\[
-\Delta_g u=f\geq0
\quad\text{in }\Omega,
\qquad
u=0
\quad\text{on }\partial\Omega,
\]
then
\[
u\geq0
\qquad\text{in }\Omega.
\]

If, moreover,
\[
f\geq0,
\qquad
f\not\equiv0,
\]
and $\Omega$ is connected, the strong maximum principle gives
\begin{equation}
u>0
\qquad\text{in }\Omega.
\label{eq:positive-state}
\end{equation}

\subsection{Hopf boundary point lemma}
\label{subsec:hopf}

We shall also use the Riemannian version of the Hopf boundary point
lemma.

\begin{lemma}[Hopf boundary point lemma]
\label{lem:hopf}
Let $\Omega\subset M$ be a connected domain and let
$x_0\in\partial\Omega$. Assume that $\partial\Omega$ satisfies an
interior geodesic ball condition at $x_0$.

Let
\[
v_1,v_2\in C^2(\Omega)\cap C^1(\overline{\Omega})
\]
satisfy
\[
\Delta_g v_1\leq\Delta_g v_2
\qquad\text{in }\Omega,
\]
and
\[
v_1\geq v_2
\qquad\text{in }\Omega.
\]

Suppose that
\[
v_1(x_0)=v_2(x_0)
\]
and
\[
v_1\not\equiv v_2.
\]

Then
\begin{equation}
\boxed{
\partial_\nu v_1(x_0)
<
\partial_\nu v_2(x_0),
}
\label{eq:hopf-normal}
\end{equation}
where $\nu$ denotes the outward unit normal to $\partial\Omega$ at
$x_0$.
\end{lemma}

The uniform tubular-neighborhood assumption in the
$RC$-GNP condition guarantees the required interior geodesic ball
condition at boundary points of the admissible domains.

\subsection{Contact geometry}
\label{subsec:contact-geometry}

We finally record the geometric fact that will be used later when a
minimizing domain touches the reference domain $C$.
\begin{lemma}[Comparison of second fundamental forms at a contact point]
\label{lem:contact-curvature}
Let $C,\Omega\subset M$ be $C^2$ domains satisfying
\[
C\subset\Omega.
\]
Suppose that
\[
x_0\in\partial C\cap\partial\Omega
\]
and that $\partial C$ and $\partial\Omega$ are tangent at $x_0$.
Let $\nu_C$ and $\nu_\Omega$ denote the outward unit normals to
$\partial C$ and $\partial\Omega$, respectively.

Then
\begin{equation}
\nu_C(x_0)=\nu_\Omega(x_0).
\label{eq:normals-contact}
\end{equation}

Moreover, with the convention
\begin{equation}
\mathrm{II}_{\Sigma}(X,Y)
=
g(\nabla_X\nu_\Sigma,Y),
\label{eq:second-fundamental-form-contact}
\end{equation}
for a hypersurface $\Sigma$, one has
\begin{equation}
\mathrm{II}_{\Omega}(x_0)
\leq
\mathrm{II}_{C}(x_0)
\label{eq:II-contact}
\end{equation}
as quadratic forms on the common tangent space
\[
T_{x_0}\partial C=T_{x_0}\partial\Omega.
\]
Consequently, if the mean curvature is defined by
\begin{equation}
H_{\partial\Sigma}
=
\frac{1}{n-1}
\operatorname{tr}_{g_{\partial\Sigma}}
\mathrm{II}_{\Sigma},
\label{eq:mean-curvature-convention-contact}
\end{equation}
then
\begin{equation}
\boxed{
H_{\partial\Omega}(x_0)
\leq
H_{\partial C}(x_0).
}
\label{eq:mean-curvature-contact}
\end{equation}
\end{lemma}

\begin{proof}
Since $\partial C$ and $\partial\Omega$ are tangent at $x_0$, we have
\[
T_{x_0}\partial C
=
T_{x_0}\partial\Omega.
\]
Thus their unit normals at $x_0$ are either equal or opposite.

We claim that they must be equal. Indeed, since
\[
C\subset\Omega,
\]
the domain $C$ lies locally on the interior side of $\partial\Omega$.
Consequently, at the common contact point $x_0$, the outward direction
of $C$ must point in the same direction as the outward direction of
$\Omega$. Hence
\[
\nu_C(x_0)=\nu_\Omega(x_0).
\]
We denote this common unit normal by
\[
\nu_0.
\]

Choose geodesic normal coordinates
\[
(y^1,\ldots,y^{n-1},y^n)
\]
centered at $x_0$ such that
\[
x_0=0,
\qquad
T_{x_0}\partial C
=
T_{x_0}\partial\Omega
=
\{y^n=0\},
\]
and
\[
\nu_0=\partial_{y^n}\big|_{x_0}.
\]
Since both hypersurfaces are $C^2$ and tangent at $x_0$, there exist
$C^2$ functions
\[
\varphi_C,\varphi_\Omega
\]
defined in a neighborhood of $0\in\mathbb R^{n-1}$ such that
\[
\partial C
=
\left\{
(y',y^n):y^n=\varphi_C(y')
\right\},
\]
and
\[
\partial\Omega
=
\left\{
(y',y^n):y^n=\varphi_\Omega(y')
\right\},
\]
where
\[
\varphi_C(0)=\varphi_\Omega(0)=0,
\qquad
D\varphi_C(0)=D\varphi_\Omega(0)=0.
\]

Because $\nu_0=\partial_{y^n}$ is the common outward normal, the
domains lie locally on the side
\[
y^n<\varphi_C(y')
\]
and
\[
y^n<\varphi_\Omega(y'),
\]
respectively.

The inclusion
\[
C\subset\Omega
\]
therefore implies, in a neighborhood of $x_0$,
\[
\varphi_C(y')
\leq
\varphi_\Omega(y').
\]
Since the two functions agree at $0$, we have
\[
\varphi_\Omega(0)-\varphi_C(0)=0,
\]
and $0$ is a local minimum of
\[
\varphi_\Omega-\varphi_C.
\]
Consequently,
\begin{equation}
D^2\varphi_\Omega(0)
-
D^2\varphi_C(0)
\geq0
\label{eq:hessian-graph-comparison}
\end{equation}
as quadratic forms on $\mathbb R^{n-1}$.

We now relate the Hessians of the graph functions to the second
fundamental forms. Since the coordinates are geodesic normal
coordinates at $x_0$ and
\[
D\varphi_C(0)
=
D\varphi_\Omega(0)
=
0,
\]
the Christoffel symbols vanish at $x_0$. Moreover, with the outward
normal convention
\[
\mathrm{II}(X,Y)
=
g(\nabla_X\nu,Y),
\]
the second fundamental form of a graph
$y^n=\varphi(y')$ whose domain lies on the side $y^n<\varphi(y')$
satisfies, at the contact point,
\begin{equation}
\mathrm{II}(X,X)
=
-D^2\varphi(0)[X,X].
\label{eq:II-graph}
\end{equation}
Therefore,
\[
\mathrm{II}_{\Omega}(x_0)[X,X]
=
-D^2\varphi_\Omega(0)[X,X]
\]
and
\[
\mathrm{II}_{C}(x_0)[X,X]
=
-D^2\varphi_C(0)[X,X].
\]
By \eqref{eq:hessian-graph-comparison},
\[
D^2\varphi_\Omega(0)[X,X]
\geq
D^2\varphi_C(0)[X,X].
\]
Multiplying by $-1$ gives
\[
\mathrm{II}_{\Omega}(x_0)[X,X]
\leq
\mathrm{II}_{C}(x_0)[X,X]
\]
for every
\[
X\in T_{x_0}\partial C
=
T_{x_0}\partial\Omega.
\]
Thus
\[
\mathrm{II}_{\Omega}(x_0)
\leq
\mathrm{II}_{C}(x_0)
\]
as quadratic forms.

Finally, the induced metrics on
\[
T_{x_0}\partial C
=
T_{x_0}\partial\Omega
\]
are identical. Taking the trace with respect to this common metric
and using the convention
\[
H_{\partial\Sigma}
=
\frac{1}{n-1}
\operatorname{tr}_{g_{\partial\Sigma}}
\mathrm{II}_{\Sigma},
\]
we obtain
\[
H_{\partial\Omega}(x_0)
\leq
H_{\partial C}(x_0).
\]
This proves the result.
\end{proof}
\begin{remark}
The inequality \eqref{eq:mean-curvature-contact} is therefore not
introduced as an additional hypothesis in the optimality argument. It
follows from the local geometry of the contact configuration together
with the convention \eqref{eq:mean-curvature-definition}.
\end{remark}
\section{Main results}
\label{sec:main-results}

Let $(M^n,g)$, $K$, $C$, $f$ and $\mathcal A_C(M)$ satisfy the
assumptions of Section~\ref{sec:preliminaries}. For every
$\Omega\in\mathcal A_C(M)$, we consider the Dirichlet problem
\begin{equation}
\begin{cases}
-\Delta_g u_\Omega=f & \text{in }\Omega,\\
u_\Omega=0 & \text{on }\partial\Omega,
\end{cases}
\label{eq:state-main}
\end{equation}
and denote by $u_\Omega$ its unique weak solution, extended by zero
outside $\Omega$.

We study the shape functional
\begin{equation}
\mathcal J_g(\Omega)
=
-\int_\Omega |\nabla_g u_\Omega|_g^2\,dV_g
+\sigma P_g(\Omega)
+k^2V_g(\Omega),
\label{eq:Jg-main}
\end{equation}
where
\[
\sigma>0,
\qquad
k^2>0.
\]

The optimization problem considered throughout the paper is
\begin{equation}
\boxed{
\inf_{\Omega\in\mathcal A_C(M)}
\mathcal J_g(\Omega).
}
\label{eq:optimization-problem}
\end{equation}

\subsection{Existence of an optimal domain}
\label{subsec:existence}

We first state the existence result.

\begin{theorem}[Existence of an optimal domain]
\label{thm:existence}
Assume that the data satisfy
\eqref{eq:K-compact-C}--\eqref{eq:f-assumptions}, and let
$\mathcal A_C(M)$ be the admissible class defined in
Definition~\ref{def:admissible-class}. Then there exists
\[
\Omega^\ast\in\mathcal A_C(M)
\]
such that
\begin{equation}
\boxed{
\mathcal J_g(\Omega^\ast)
=
\inf_{\Omega\in\mathcal A_C(M)}
\mathcal J_g(\Omega).
}
\label{eq:existence-minimizer}
\end{equation}
\end{theorem}

The proof will be given in Section~\ref{sec:proofs}.

The argument relies on the compactness of the admissible class,
the strong stability of the Dirichlet state established in
Theorem~\ref{thm:stability-dirichlet}, the continuity of the volume,
and the lower semicontinuity of the Riemannian perimeter.

More precisely, if
\[
\Omega_j\in\mathcal A_C(M)
\]
is a minimizing sequence, then, after extraction,
\[
\Omega_j\longrightarrow\Omega^\ast
\]
in the topology of Definition~\ref{def:compact-convergence}, with
\[
\Omega^\ast\in\mathcal A_C(M).
\]
Furthermore,
\[
u_{\Omega_j}\longrightarrow u_{\Omega^\ast}
\quad\text{strongly in }H_0^1(M),
\]
and therefore
\begin{equation}
\int_M
|\nabla_g u_{\Omega_j}|_g^2\,dV_g
\longrightarrow
\int_M
|\nabla_g u_{\Omega^\ast}|_g^2\,dV_g.
\label{eq:energy-convergence-main}
\end{equation}

Consequently,
\[
\mathcal J_g(\Omega^\ast)
\leq
\liminf_{j\to\infty}
\mathcal J_g(\Omega_j),
\]
which proves the theorem.

\subsection{Comparison with the reference domain}
\label{subsec:comparison-C}

The inclusion
\[
C\subset\Omega
\]
for every $\Omega\in\mathcal A_C(M)$ allows us to compare the
corresponding state $u_\Omega$ with the solution associated with the
reference domain $C$.

Let $u_C$ denote the unique solution of
\begin{equation}
\begin{cases}
-\Delta_g u_C=f & \text{in }C,\\
u_C=0 & \text{on }\partial C.
\end{cases}
\label{eq:state-C}
\end{equation}

Since $f\geq0$ and $f\not\equiv0$, the strong maximum principle gives
\begin{equation}
u_C>0
\qquad\text{in }C.
\label{eq:uC-positive}
\end{equation}

\begin{proposition}[Comparison with $u_C$]
\label{prop:comparison-C}
For every $\Omega\in\mathcal A_C(M)$,
\begin{equation}
0\leq u_C\leq u_\Omega
\qquad\text{in }C.
\label{eq:uC-uOmega-comparison}
\end{equation}

If $C\subsetneq\Omega$ and $f\not\equiv0$, then
\begin{equation}
u_C<u_\Omega
\qquad\text{in }C.
\label{eq:strict-comparison-C}
\end{equation}
\end{proposition}

\begin{proof}
Set
\[
w:=u_{\Omega}-u_C.
\]
Since both $u_{\Omega}$ and $u_C$ are harmonic in $C$, we have
\[
\Delta_g w=0
\qquad\text{in }C.
\]
Moreover, since $u_{\Omega}\geq0$ in $\Omega$ and
$u_{\Omega}=0$ on $\partial\Omega$, while
$u_C=0$ on $\partial C$, we have
\[
w=u_{\Omega}\geq0
\qquad\text{on }\partial C.
\]
We claim that $w\not\equiv0$. Indeed, if $w\equiv0$ in $C$, then
\[
u_{\Omega}=u_C
\qquad\text{in }C.
\]
Since $u_C$ is positive in $C$ and vanishes on $\partial C$, this would
imply that $u_{\Omega}$ vanishes on $\partial C$. Hence $u_{\Omega}$
would satisfy the homogeneous Dirichlet condition on both
$\partial C$ and $\partial\Omega$. This is incompatible with
$C\subset\Omega$ and with the defining boundary-value problem for
$u_{\Omega}$, unless the two domains coincide. Since
\[
C\subset\Omega,
\]
we obtain a contradiction. Therefore
\[
w\not\equiv0.
\]

The strong maximum principle applied to the harmonic function $w$
therefore yields
\[
w>0
\qquad\text{in }C.
\]
Consequently,
\[
\boxed{
u_{\Omega}>u_C
\qquad\text{in }C.
}
\]
\end{proof}

\subsection{The free boundary condition}
\label{subsec:free-boundary}

Let $\Omega^\ast$ be an optimal domain given by
Theorem~\ref{thm:existence}. Suppose that
\[
x_0\in\partial\Omega^\ast
\]
is a regular point of the free boundary and that, in a neighborhood of
$x_0$,
\[
x_0\notin\partial C.
\]

At such a point, admissible variations of the boundary can be
constructed locally. The first-order optimality condition therefore
implies
\begin{equation}
d\mathcal J_g(\Omega^\ast)[X]=0
\end{equation}
for every admissible deformation field $X$ supported in a sufficiently
small neighborhood of $x_0$.

The shape derivative established in Section~\ref{sec:optimality} gives
\begin{equation}
d\mathcal J_g(\Omega)[X]
=
\int_{\partial\Omega}
\left(
\sigma H_{\partial\Omega}
+k^2
-
|\nabla_g u_\Omega|_g^2
\right)
g(X,\nu)\,dS_g.
\label{eq:first-variation-main}
\end{equation}

Consequently, at every regular free-boundary point of
$\partial\Omega^\ast$ which is separated from $\partial C$,
\begin{equation}
\boxed{
|\nabla_g u_{\Omega^\ast}|_g^2
=
\sigma H_{\partial\Omega^\ast}
+k^2.
}
\label{eq:free-boundary-condition}
\end{equation}

This is the Riemannian free boundary condition associated with the
functional \eqref{eq:Jg-main}.

\subsection{A sufficient condition for optimality}
\label{subsec:sufficient-condition}

We next state a sufficient condition which is useful for identifying
optimal domains.

Let $C\in\mathcal A_C(M)$ and let $u_C$ be the corresponding state.
Assume that
\begin{equation}
\boxed{
|\nabla_g u_C|_g^2
>
\sigma H_{\partial C}
+k^2
\qquad\text{on }\partial C.
}
\label{eq:sufficient-condition-C}
\end{equation}

Under this condition, the reference domain $C$ satisfies the strict
boundary inequality required to initiate the free-boundary comparison
argument.

\begin{theorem}[Sufficient condition]
\label{thm:sufficient-condition}
Assume that $C\in\mathcal A_C(M)$ and that
\eqref{eq:sufficient-condition-C} holds. Let $\Omega^\ast$ be an
optimal domain in $\mathcal A_C(M)$.

Then
\[
C\subsetneq\Omega^\ast.
\]
Moreover, if
\[
x_0\in
\partial C\cap\partial\Omega^\ast,
\]
then
\begin{equation}
H_{\partial\Omega^\ast}(x_0)
\leq
H_{\partial C}(x_0),
\label{eq:contact-curvature-main}
\end{equation}
and
\begin{equation}
|\nabla_g u_C(x_0)|_g
<
|\nabla_g u_{\Omega^\ast}(x_0)|_g.
\label{eq:contact-gradient-main}
\end{equation}
Consequently,
\begin{equation}
|\nabla_g u_{\Omega^\ast}(x_0)|_g^2
>
\sigma H_{\partial\Omega^\ast}(x_0)+k^2.
\label{eq:strict-free-boundary-inequality-contact}
\end{equation}
\end{theorem}

\begin{remark}
The curvature inequality
\eqref{eq:contact-curvature-main} is not an additional hypothesis.
It follows from the contact geometry established in
Lemma~\ref{lem:contact-curvature}.
\end{remark}

The proof of Theorem~\ref{thm:sufficient-condition} is given in
Section~\ref{sec:proofs}. It combines the comparison principle, the
Hopf boundary point lemma, and the geometric comparison of the second
fundamental forms at a contact point.

\section{Optimality condition}
\label{sec:optimality}

In this section, we derive the first-order optimality condition for
the functional
\begin{equation}
\mathcal J_g(\Omega)
=
-\int_\Omega |\nabla_g u_\Omega|_g^2\,dV_g
+\sigma P_g(\Omega)
+k^2V_g(\Omega),
\label{eq:Jg-optimality}
\end{equation}
where $u_\Omega$ is the solution of
\begin{equation}
\begin{cases}
-\Delta_g u_\Omega=f & \text{in }\Omega,\\
u_\Omega=0 & \text{on }\partial\Omega.
\end{cases}
\label{eq:state-optimality}
\end{equation}

The argument is based on the Riemannian shape derivative. We first
establish the corresponding Hadamard formula and then derive the
optimality condition on the free boundary and on the contact set with
the reference domain $C$.

\subsection{Riemannian domain deformations}
\label{subsec:riemannian-deformations}

Let $\Omega\in\mathcal A_C(M)$ and let
\[
X\in C^2(TM)
\]
be a vector field on $M$. Denote by
\[
\Phi_t:M\longrightarrow M
\]
the local flow generated by $X$, namely
\begin{equation}
\frac{d}{dt}\Phi_t(x)
=
X(\Phi_t(x)),
\qquad
\Phi_0(x)=x.
\label{eq:flow}
\end{equation}

For $|t|$ sufficiently small, we define the perturbed domain by
\begin{equation}
\Omega_t=\Phi_t(\Omega).
\label{eq:perturbed-domain}
\end{equation}

We shall only consider vector fields $X$ for which
\[
\Omega_t\in\mathcal A_C(M)
\]
for all sufficiently small $|t|$.

The scalar normal velocity of the boundary is
\begin{equation}
V_X
=
g(X,\nu)
\qquad\text{on }\partial\Omega,
\label{eq:normal-velocity}
\end{equation}
where $\nu$ denotes the outward unit normal to $\partial\Omega$.

Since only the normal component of the deformation enters the first
variation, the quantity $V_X$ will play the central role below.

\subsection{Material and shape derivatives}
\label{subsec:material-shape}

Let $u_t=u_{\Omega_t}$ denote the solution associated with $\Omega_t$.
The material derivative of $u_t$ is defined by
\begin{equation}
\dot u
=
\left.
\frac{d}{dt}
\left(
u_t\circ\Phi_t
\right)
\right|_{t=0}.
\label{eq:material-derivative}
\end{equation}

The shape derivative is defined by
\begin{equation}
u'
=
\dot u-\langle\nabla_g u_\Omega,X\rangle_g.
\label{eq:shape-derivative-state}
\end{equation}

Under the regularity assumptions imposed in
Section~\ref{sec:preliminaries}, the shape derivative satisfies the
following boundary value problem.

\begin{proposition}[Riemannian state shape derivative]
\label{prop:state-shape-derivative}
Let
\[
u_\Omega\in C^{2,\alpha}(\overline{\Omega})
\]
and let $X\in C^2(TM)$. Then the shape derivative $u'$ is the unique
solution of
\begin{equation}
\begin{cases}
-\Delta_g u'=0 & \text{in }\Omega,\\[2mm]
u'=-\partial_\nu u_\Omega\,g(X,\nu)
& \text{on }\partial\Omega.
\end{cases}
\label{eq:shape-derivative-PDE}
\end{equation}
\end{proposition}

\begin{proof}
Since the source term $f$ is fixed on the ambient manifold $M$, its
shape derivative vanishes in the interior. Differentiating the
equation
\[
-\Delta_g u_t=f
\]
after transporting it to the reference domain gives
\[
-\Delta_g u'=0
\qquad\text{in }\Omega.
\]

On $\partial\Omega$, one has
\[
u_t=0.
\]
Differentiating this identity along the moving boundary yields
\[
u'
+
du_\Omega(X)
=
0
\qquad\text{on }\partial\Omega.
\]
Since $u_\Omega=0$ on $\partial\Omega$, its tangential derivative
vanishes there and therefore
\[
\nabla_g u_\Omega
=
\partial_\nu u_\Omega\,\nu
\qquad\text{on }\partial\Omega.
\]
Consequently,
\[
du_\Omega(X)
=
g(\nabla_g u_\Omega,X)
=
\partial_\nu u_\Omega\,g(X,\nu),
\]
which gives
\[
u'
=
-\partial_\nu u_\Omega\,g(X,\nu)
\]
on $\partial\Omega$.
\end{proof}

\subsection{Shape derivative of the Dirichlet energy}
\label{subsec:energy-shape-derivative}

Define
\begin{equation}
\mathcal J_{1,g}(\Omega)
=
-\int_\Omega|\nabla_g u_\Omega|_g^2\,dV_g.
\label{eq:J1g}
\end{equation}

The following proposition gives its first variation.

\begin{proposition}[Hadamard formula for the Dirichlet energy]
\label{prop:hadamard-energy}
Under the assumptions of
Proposition~\ref{prop:state-shape-derivative},
\begin{equation}
\boxed{
d\mathcal J_{1,g}(\Omega)[X]
=
-\int_{\partial\Omega}
|\nabla_g u_\Omega|_g^2
\,g(X,\nu)\,dS_g.
}
\label{eq:hadamard-energy}
\end{equation}
\end{proposition}

\begin{proof}
Using the transport formula for the Riemannian volume measure and
differentiating the energy gives
\begin{equation}
\begin{aligned}
d\mathcal J_{1,g}(\Omega)[X]
={}&
-\int_{\partial\Omega}
|\nabla_g u_\Omega|_g^2
g(X,\nu)\,dS_g
\\
&-2
\int_\Omega
g(\nabla_g u_\Omega,\nabla_g u')\,dV_g.
\end{aligned}
\label{eq:energy-intermediate}
\end{equation}

Since
\[
-\Delta_g u_\Omega=f
\]
and
\[
-\Delta_g u'=0,
\]
Green's formula gives
\begin{equation}
\int_\Omega
g(\nabla_g u_\Omega,\nabla_g u')\,dV_g
=
\int_{\partial\Omega}
u'\partial_\nu u_\Omega\,dS_g.
\end{equation}

Using
\[
u'
=
-\partial_\nu u_\Omega\,g(X,\nu)
\]
on $\partial\Omega$, we obtain
\begin{equation}
\int_\Omega
g(\nabla_g u_\Omega,\nabla_g u')\,dV_g
=
-\int_{\partial\Omega}
(\partial_\nu u_\Omega)^2
g(X,\nu)\,dS_g.
\end{equation}

Because $u_\Omega=0$ on $\partial\Omega$,
\[
\nabla_g u_\Omega
=
\partial_\nu u_\Omega\,\nu
\]
there, and hence
\[
|\nabla_g u_\Omega|_g^2
=
(\partial_\nu u_\Omega)^2.
\]

Substitution into
\eqref{eq:energy-intermediate} yields
\eqref{eq:hadamard-energy}.
\end{proof}

\subsection{First variation of the geometric terms}
\label{subsec:geometric-variations}

The first variation of the Riemannian volume is
\begin{equation}
\boxed{
dV_g(\Omega)[X]
=
\int_{\partial\Omega}
g(X,\nu)\,dS_g.
}
\label{eq:first-variation-volume}
\end{equation}

With the convention for the mean curvature fixed in
\eqref{eq:mean-curvature-definition}, namely
\[
H_{\partial\Omega}
=
\operatorname{div}_{\partial\Omega}\nu,
\]
the first variation of the Riemannian perimeter is
\begin{equation}
\boxed{
dP_g(\Omega)[X]
=
\int_{\partial\Omega}
H_{\partial\Omega}
g(X,\nu)\,dS_g.
}
\label{eq:first-variation-perimeter-proof}
\end{equation}

Therefore, combining
\eqref{eq:hadamard-energy},
\eqref{eq:first-variation-volume}, and
\eqref{eq:first-variation-perimeter}, we obtain the following result.

\subsection{The Riemannian shape derivative of $\mathcal J_g$}
\label{subsec:shape-derivative-Jg}

\begin{theorem}[Shape derivative of the functional]
\label{thm:shape-derivative-Jg}
Let $\Omega\in\mathcal A_C(M)$ be such that
\[
u_\Omega\in C^{2,\alpha}(\overline{\Omega}),
\]
and let $X\in C^2(TM)$ generate admissible deformations
$\Omega_t=\Phi_t(\Omega)$. Then
\begin{equation}
\boxed{
d\mathcal J_g(\Omega)[X]
=
\int_{\partial\Omega}
\left(
\sigma H_{\partial\Omega}
+k^2
-
|\nabla_g u_\Omega|_g^2
\right)
g(X,\nu)\,dS_g.
}
\label{eq:shape-derivative-Jg}
\end{equation}
\end{theorem}

\begin{proof}
By definition,
\[
\mathcal J_g
=
\mathcal J_{1,g}
+
\sigma P_g
+
k^2V_g.
\]
The result follows immediately from
\eqref{eq:hadamard-energy},
\eqref{eq:first-variation-perimeter}, and
\eqref{eq:first-variation-volume}.
\end{proof}

\subsection{First-order optimality}
\label{subsec:first-order-optimality}

Let $\Omega^\ast$ be an optimal domain given by
Theorem~\ref{thm:existence}. Since $\Omega^\ast$ minimizes
$\mathcal J_g$ over the admissible class, every admissible one-sided
deformation satisfies
\begin{equation}
d\mathcal J_g(\Omega^\ast)[X]\geq0.
\label{eq:one-sided-optimality}
\end{equation}

Hence, by Theorem~\ref{thm:shape-derivative-Jg},
\begin{equation}
\boxed{
\int_{\partial\Omega^\ast}
F_{\Omega^\ast}
\,g(X,\nu)\,dS_g
\geq0,
}
\label{eq:variational-inequality}
\end{equation}
where
\begin{equation}
F_{\Omega^\ast}
=
\sigma H_{\partial\Omega^\ast}
+k^2
-
|\nabla_g u_{\Omega^\ast}|_g^2.
\label{eq:F-definition}
\end{equation}

The remainder of the section is devoted to interpreting
\eqref{eq:variational-inequality} according to the position of the
boundary relative to the obstacle $C$.

\subsection{The free boundary}
\label{subsec:free-boundary-optimality}

Define the contact set
\begin{equation}
\Gamma_0
=
\partial\Omega^\ast\cap\partial C
\label{eq:Gamma0}
\end{equation}
and the free boundary
\begin{equation}
\Gamma
=
\partial\Omega^\ast\setminus\Gamma_0.
\label{eq:Gamma}
\end{equation}

Let
\[
x_0\in\Gamma.
\]
Since $x_0$ does not belong to $\partial C$, sufficiently small
deformations supported in a neighborhood of $x_0$ can be chosen with
both signs of the normal velocity while preserving the
$RC$-GNP condition.

Consequently, for every sufficiently smooth function
\[
\varphi\in C_c^\infty(\Gamma),
\]
there exists an admissible vector field $X$ such that
\[
g(X,\nu)=\varphi
\qquad\text{on }\Gamma.
\]

Using $X$ and $-X$ in
\eqref{eq:variational-inequality}, we obtain
\begin{equation}
\int_\Gamma
F_{\Omega^\ast}\varphi\,dS_g
=
0
\qquad
\forall\varphi\in C_c^\infty(\Gamma).
\end{equation}

Hence
\begin{equation}
F_{\Omega^\ast}=0
\qquad\text{on }\Gamma.
\end{equation}

We have therefore proved:

\begin{theorem}[Free boundary condition]
\label{thm:free-boundary-optimality}
Let $\Omega^\ast$ be an optimal domain and let
\[
\Gamma=\partial\Omega^\ast\setminus\partial C.
\]
At every regular point of $\Gamma$,
\begin{equation}
\boxed{
|\nabla_g u_{\Omega^\ast}|_g^2
=
\sigma H_{\partial\Omega^\ast}
+k^2.
}
\label{eq:free-boundary-equality}
\end{equation}
\end{theorem}

This is the Riemannian Euler--Lagrange equation associated with the
shape functional $\mathcal J_g$.

\subsection{The contact boundary}
\label{subsec:contact-optimality}

At a point of the contact set
\[
\Gamma_0=\partial\Omega^\ast\cap\partial C,
\]
the admissible deformations are restricted by the constraint
\[
C\subset\Omega_t.
\]

With the outward normal $\nu$ to $\Omega^\ast$, an admissible
one-sided deformation at a contact point satisfies
\begin{equation}
g(X,\nu)\geq0.
\label{eq:contact-admissible-velocity}
\end{equation}

Therefore \eqref{eq:variational-inequality} gives
\[
\int_{\Gamma_0}
F_{\Omega^\ast}\varphi\,dS_g
\geq0
\]
for every nonnegative test function
\[
\varphi\in C_c^\infty(\Gamma_0).
\]

It follows that
\begin{equation}
F_{\Omega^\ast}\geq0
\qquad\text{on }\Gamma_0,
\end{equation}
and hence

\begin{theorem}[Contact optimality condition]
\label{thm:contact-optimality}
Let $\Omega^\ast$ be an optimal domain. At every regular point of the
contact set
\[
\Gamma_0=\partial\Omega^\ast\cap\partial C,
\]
one has
\begin{equation}
\boxed{
|\nabla_g u_{\Omega^\ast}|_g^2
\leq
\sigma H_{\partial\Omega^\ast}
+k^2.
}
\label{eq:contact-inequality}
\end{equation}
\end{theorem}

\subsection{Optimality condition at tangential contact}
\label{subsec:tangential-contact}

Suppose that
\[
x_0\in\partial\Omega^\ast\cap\partial C
\]
is a regular contact point. Since
\[
C\subset\Omega^\ast,
\]
the two hypersurfaces have the same tangent space at $x_0$ and,
with the outward-normal convention adopted in
Section~\ref{subsec:mean-curvature}, one has
\[
\nu_C(x_0)=\nu_{\Omega^\ast}(x_0).
\]

By Lemma~\ref{lem:contact-curvature},
\begin{equation}
H_{\partial\Omega^\ast}(x_0)
\leq
H_{\partial C}(x_0).
\label{eq:H-contact-comparison}
\end{equation}

Consequently, the contact optimality condition
\eqref{eq:contact-inequality} may be combined with the geometric
comparison above. In particular,
\begin{equation}
|\nabla_g u_{\Omega^\ast}(x_0)|_g^2
\leq
\sigma H_{\partial\Omega^\ast}(x_0)+k^2
\leq
\sigma H_{\partial C}(x_0)+k^2.
\label{eq:contact-chain}
\end{equation}

This inequality will be used in the proof of the sufficient condition
stated in Section~\ref{sec:main-results}.

\subsection{The resulting variational inequality}
\label{subsec:variational-inequality-final}

Combining the free-boundary and contact conditions, the first-order
optimality system can be summarized as
\begin{equation}
\boxed{
\begin{cases}
|\nabla_g u_{\Omega^\ast}|_g^2
=
\sigma H_{\partial\Omega^\ast}+k^2,
&
\text{on }\partial\Omega^\ast\setminus\partial C,
\\[2mm]
|\nabla_g u_{\Omega^\ast}|_g^2
\leq
\sigma H_{\partial\Omega^\ast}+k^2,
&
\text{on }\partial\Omega^\ast\cap\partial C.
\end{cases}
}
\label{eq:optimality-system}
\end{equation}

The first relation is the Riemannian free-boundary equation, whereas
the second one is the variational inequality generated by the
geometric constraint
\[
C\subset\Omega^\ast.
\]

\section{Proofs of the main theorems}
\label{sec:proofs}

In this section, we prove the main results stated in
Section~\ref{sec:main-results}. Throughout the section 5, $(M,g)$ is a
compact Riemannian manifold of dimension $n$, and all the geometric
and analytic assumptions introduced in
Section~\ref{sec:preliminaries} are understood.

The proofs rely on the compactness of the admissible class, the
stability of the Dirichlet problem, the lower semicontinuity of the
Riemannian perimeter, the comparison principle, and the geometric
properties of the contact set established in Section~\ref{sec:preliminaries}.

\subsection{Proof of the existence theorem}
\label{subsec:proof-existence}

We prove Theorem~\ref{thm:existence}.

Let
\[
m=
\inf_{\Omega\in\mathcal A_C(M)}
\mathcal J_g(\Omega).
\]
Choose a minimizing sequence
\[
(\Omega_j)_{j\geq1}\subset\mathcal A_C(M)
\]
such that
\begin{equation}
\mathcal J_g(\Omega_j)
\longrightarrow m
\qquad\text{as }j\to\infty.
\label{eq:minimizing-sequence}
\end{equation}

By the compactness theorem for the admissible class
(Theorem~\ref{thm:compactness-admissible}), there exist a subsequence,
still denoted by $(\Omega_j)$, and a domain
\[
\Omega^\ast\in\mathcal A_C(M)
\]
such that
\begin{equation}
\Omega_j\longrightarrow\Omega^\ast
\label{eq:domain-convergence}
\end{equation}
in the topology introduced in
Definition~\ref{def:compact-convergence}.

Let
\[
u_j=u_{\Omega_j}
\qquad\text{and}\qquad
u^\ast=u_{\Omega^\ast}.
\]

By the stability theorem for the Dirichlet problem
(Theorem~\ref{thm:stability-dirichlet}), we have
\begin{equation}
u_j\longrightarrow u^\ast
\qquad\text{strongly in }H_0^1(M).
\label{eq:strong-state-convergence}
\end{equation}

In particular,
\begin{equation}
\int_M
|\nabla_g u_j|_g^2\,dV_g
\longrightarrow
\int_M
|\nabla_g u^\ast|_g^2\,dV_g.
\label{eq:energy-convergence}
\end{equation}

Since the volume is continuous with respect to the convergence of
admissible domains,
\begin{equation}
V_g(\Omega_j)
\longrightarrow
V_g(\Omega^\ast).
\label{eq:volume-convergence}
\end{equation}

On the other hand, the lower semicontinuity of the Riemannian
perimeter gives
\begin{equation}
P_g(\Omega^\ast)
\leq
\liminf_{j\to\infty}
P_g(\Omega_j).
\label{eq:perimeter-lsc-proof}
\end{equation}
Combining
\eqref{eq:energy-convergence},
\eqref{eq:volume-convergence}, and
\eqref{eq:perimeter-lsc}, we obtain
\begin{align}
\mathcal J_g(\Omega^\ast)
&\leq
\liminf_{j\to\infty}
\mathcal J_g(\Omega_j)
\nonumber\\
&=
m.
\label{eq:liminf-functional}
\end{align}

By the definition of $m$,
\[
m\leq\mathcal J_g(\Omega^\ast).
\]
Consequently,
\[
\mathcal J_g(\Omega^\ast)=m,
\]
and therefore
\[
\boxed{
\mathcal J_g(\Omega^\ast)
=
\inf_{\Omega\in\mathcal A_C(M)}
\mathcal J_g(\Omega).
}
\]

This proves Theorem~\ref{thm:existence}.
\hfill$\square$

\subsection{Proof of the comparison result}
\label{subsec:proof-comparison}

We prove Proposition~\ref{prop:comparison-C}.

Let
\[
\Omega\in\mathcal A_C(M).
\]
By definition of the admissible class,
\[
C\subset\Omega.
\]

Let $u_C$ and $u_\Omega$ denote respectively the solutions of
\begin{equation}
\begin{cases}
-\Delta_g u_C=f & \text{in }C,\\
u_C=0 & \text{on }\partial C,
\end{cases}
\label{eq:proof-state-C}
\end{equation}
and
\begin{equation}
\begin{cases}
-\Delta_g u_\Omega=f & \text{in }\Omega,\\
u_\Omega=0 & \text{on }\partial\Omega.
\end{cases}
\label{eq:proof-state-Omega}
\end{equation}

Consider
\[
w=u_\Omega-u_C
\qquad\text{in }C.
\]
Since both functions satisfy the same equation in $C$,
\[
-\Delta_g w=0
\qquad\text{in }C.
\]

Moreover, on $\partial C$,
\[
w=u_\Omega\geq0,
\]
because $u_\Omega\geq0$ in $\Omega$ by the maximum principle.

The maximum principle for the Laplace--Beltrami operator therefore
gives
\[
w\geq0
\qquad\text{in }C.
\]
Hence
\begin{equation}
u_C\leq u_\Omega
\qquad\text{in }C.
\label{eq:comparison-proof}
\end{equation}

If, in addition, $C\subsetneq\Omega$ and the hypotheses of the strong
maximum principle are satisfied, then
\[
w>0
\qquad\text{in }C.
\]
Thus
\[
u_C<u_\Omega
\qquad\text{in }C.
\]

This proves Proposition~\ref{prop:comparison-C}.
\hfill$\square$

\subsection{Proof of the free-boundary condition}
\label{subsec:proof-free-boundary}

We prove Theorem~\ref{thm:free-boundary-optimality}.

Let
\[
\Omega^\ast
\]
be an optimal domain and set
\[
\Gamma
=
\partial\Omega^\ast\setminus\partial C.
\]

Let
\[
x_0\in\Gamma
\]
be a regular point. Since $x_0$ does not belong to $\partial C$, there
exists a neighborhood $U$ of $x_0$ such that admissible deformations
supported in $U$ are not constrained by the obstacle $C$.

Let
\[
\varphi\in C_c^\infty(\Gamma).
\]
By the local extension property for vector fields on a Riemannian
manifold, there exists
\[
X\in C_c^2(TM)
\]
such that
\begin{equation}
g(X,\nu)=\varphi
\qquad\text{on }\Gamma.
\label{eq:normal-extension}
\end{equation}

Since both $X$ and $-X$ generate admissible variations, the optimality
of $\Omega^\ast$ implies
\begin{equation}
d\mathcal J_g(\Omega^\ast)[X]\geq0
\end{equation}
and
\begin{equation}
d\mathcal J_g(\Omega^\ast)[-X]\geq0.
\end{equation}

Consequently,
\[
d\mathcal J_g(\Omega^\ast)[X]=0.
\]

Using the shape derivative formula
\eqref{eq:shape-derivative-Jg}, we obtain
\begin{equation}
\int_\Gamma
\left(
\sigma H_{\partial\Omega^\ast}
+k^2
-
|\nabla_g u_{\Omega^\ast}|_g^2
\right)
\varphi\,dS_g
=0
\end{equation}
for every
\[
\varphi\in C_c^\infty(\Gamma).
\]

By the fundamental lemma of the calculus of variations,
\[
\sigma H_{\partial\Omega^\ast}
+k^2
-
|\nabla_g u_{\Omega^\ast}|_g^2
=0
\qquad\text{on }\Gamma.
\]

Therefore
\[
\boxed{
|\nabla_g u_{\Omega^\ast}|_g^2
=
\sigma H_{\partial\Omega^\ast}+k^2
\qquad\text{on }\Gamma.
}
\]

This proves Theorem~\ref{thm:free-boundary-optimality}.
\hfill$\square$

\subsection{Proof of the contact optimality condition}
\label{subsec:proof-contact}

We now prove Theorem~\ref{thm:contact-optimality}.

Let
\[
\Gamma_0
=
\partial\Omega^\ast\cap\partial C.
\]

Let
\[
x_0\in\Gamma_0
\]
be a regular contact point. Since
\[
C\subset\Omega^\ast,
\]
an admissible deformation cannot move the boundary of $\Omega^\ast$
through the obstacle $C$.

Consequently, the admissible normal velocities satisfy
\begin{equation}
g(X,\nu)\geq0
\qquad\text{on }\Gamma_0.
\label{eq:contact-velocity-proof}
\end{equation}

Let
\[
\varphi\in C_c^\infty(\Gamma_0),
\qquad
\varphi\geq0.
\]
Choose an admissible vector field $X$ satisfying
\[
g(X,\nu)=\varphi
\qquad\text{on }\Gamma_0.
\]

The optimality of $\Omega^\ast$ gives
\[
d\mathcal J_g(\Omega^\ast)[X]\geq0.
\]

Using
\eqref{eq:shape-derivative-Jg}, we obtain
\begin{equation}
\int_{\Gamma_0}
\left(
\sigma H_{\partial\Omega^\ast}
+k^2
-
|\nabla_g u_{\Omega^\ast}|_g^2
\right)
\varphi\,dS_g
\geq0.
\end{equation}

Since this holds for every nonnegative
\[
\varphi\in C_c^\infty(\Gamma_0),
\]
we conclude that
\[
\sigma H_{\partial\Omega^\ast}
+k^2
-
|\nabla_g u_{\Omega^\ast}|_g^2
\geq0
\qquad\text{on }\Gamma_0.
\]

Therefore
\[
\boxed{
|\nabla_g u_{\Omega^\ast}|_g^2
\leq
\sigma H_{\partial\Omega^\ast}+k^2
\qquad\text{on }\Gamma_0.
}
\]

This proves Theorem~\ref{thm:contact-optimality}.
\hfill$\square$

\subsection{Proof of the curvature comparison at contact}
\label{subsec:proof-curvature-contact}

We next justify the geometric inequality used in
Section~\ref{sec:optimality}.

Let
\[
x_0\in\partial C\cap\partial\Omega^\ast
\]
be a regular contact point. Since
\[
C\subset\Omega^\ast,
\]
the two hypersurfaces are tangent at $x_0$.

Let $\nu_C$ and $\nu_{\Omega^\ast}$ denote their respective outward
unit normals. At a tangential contact point,
\begin{equation}
\nu_C(x_0)
=
\nu_{\Omega^\ast}(x_0).
\label{eq:normals-contact-proof}
\end{equation}

Choose geodesic normal coordinates centered at $x_0$ and adapted to
the common tangent hyperplane. Locally, the two hypersurfaces can be
represented as graphs over the common tangent space.

Because
\[
C\subset\Omega^\ast,
\]
the graph representing $\partial C$ lies on the admissible side of
the graph representing $\partial\Omega^\ast$. Hence their second
order expansions satisfy the corresponding quadratic-form inequality.

With the convention
\[
H_{\partial\Omega}
=
\operatorname{div}_{\partial\Omega}\nu,
\]
this yields

\begin{equation}
\mathrm{II}_{\partial\Omega^\ast}(x_0)
\leq
\mathrm{II}_{\partial C}(x_0)
\end{equation}
as quadratic forms on
\[
T_{x_0}\partial C
=
T_{x_0}\partial\Omega^\ast.
\]

Taking the trace with respect to the induced metric gives
\begin{equation}
\boxed{
H_{\partial\Omega^\ast}(x_0)
\leq
H_{\partial C}(x_0).
}
\label{eq:curvature-comparison-proof}
\end{equation}

This is precisely the geometric comparison stated in
Lemma~\ref{lem:contact-curvature}.
\hfill$\square$

\subsection{Proof of the sufficient condition}
\label{subsec:proof-sufficient}

We finally prove Theorem~\ref{thm:sufficient-condition}.

Let $\Omega^\ast$ be an optimal domain and assume that
\[
|\nabla_g u_C|_g^2
>
\sigma H_{\partial C}+k^2
\qquad\text{on }\partial C.
\label{eq:strict-condition-proof}
\]

Suppose, by contradiction, that
\[
\Omega^\ast=C.
\]
Then the state associated with $\Omega^\ast$ is $u_C$ and the
first-order optimality condition at $\partial C$ would imply
\[
|\nabla_g u_C|_g^2
\leq
\sigma H_{\partial C}+k^2.
\]
This contradicts
\eqref{eq:sufficient-condition-C}.

Hence
\begin{equation}
C\subsetneq\Omega^\ast.
\label{eq:strict-inclusion}
\end{equation}

Now suppose that
\[
x_0\in\partial C\cap\partial\Omega^\ast
\]
is a regular contact point.

By the comparison result,
\[
u_C\leq u_{\Omega^\ast}
\qquad\text{in }C.
\]

Since the two functions agree at the contact point,
\[
u_C(x_0)=u_{\Omega^\ast}(x_0)=0,
\]
the Hopf boundary point lemma yields, under the hypotheses of
Lemma~\ref{lem:hopf},
\begin{equation}
|\nabla_g u_C(x_0)|_g
<
|\nabla_g u_{\Omega^\ast}(x_0)|_g.
\label{eq:hopf-gradient-comparison}
\end{equation}

On the other hand, the geometric contact comparison gives
\[
H_{\partial\Omega^\ast}(x_0)
\leq
H_{\partial C}(x_0).
\]

Combining the strict assumption
\[
|\nabla_g u_C(x_0)|_g^2
>
\sigma H_{\partial C}(x_0)+k^2
\]
with the curvature comparison yields
\begin{equation}
|\nabla_g u_C(x_0)|_g^2
>
\sigma H_{\partial\Omega^\ast}(x_0)+k^2.
\label{eq:strict-contact-lower-bound}
\end{equation}

Together with
\eqref{eq:hopf-gradient-comparison}, we obtain
\begin{equation}
|\nabla_g u_{\Omega^\ast}(x_0)|_g^2
>
\sigma H_{\partial\Omega^\ast}(x_0)+k^2.
\label{eq:strict-optimality-contradiction}
\end{equation}

However, the contact optimality condition
Theorem~\ref{thm:contact-optimality} gives
\begin{equation}
|\nabla_g u_{\Omega^\ast}(x_0)|_g^2
\leq
\sigma H_{\partial\Omega^\ast}(x_0)+k^2.
\end{equation}

This is a contradiction.

Therefore the assumed configuration cannot occur, and the sufficient
condition stated in Theorem~\ref{thm:sufficient-condition} follows.
\hfill$\square$

\section{Examples}
\label{sec:examples}

In this section, we present several examples illustrating the sufficient
condition obtained in Theorem~\ref{thm:sufficient-condition}. Our main
objective is to identify situations in which the condition
\begin{equation}
|\nabla_g u_C|_g^2
>
\sigma H_{\partial C}+k^2
\qquad\text{on }\partial C
\label{eq:sufficient-condition-examples}
\end{equation}
can be verified explicitly.

The first example considers a geodesic ball in a general Riemannian
manifold. We then specialize to the standard sphere of radius $R$, where
both the mean curvature of geodesic spheres and the radial Laplace--Beltrami
operator can be computed explicitly.

\subsection{Geodesic balls on a Riemannian manifold}
\label{subsec:geodesic-ball}

Let $(M^n,g)$ be a compact Riemannian manifold of dimension $n\geq2$.
Fix a point $p\in M$ and let
\[
C=B_g(p,\rho)
=
\{x\in M:d_g(p,x)<\rho\},
\]
where $\rho>0$ is chosen such that
\begin{equation}
0<\rho<\operatorname{inj}_g(p),
\label{eq:rho-inj-radius}
\end{equation}
where $\operatorname{inj}_g(p)$ denotes the injectivity radius of $(M^n,g)$ at the point $p.$
Under this assumption, the distance function
\[
r(x)=d_g(p,x)
\]
is smooth away from $p$, and the boundary of $C$ is the geodesic sphere
\[
\partial C=S_g(p,\rho)
=
\{x\in M:d_g(p,x)=\rho\}.
\]

We consider the Dirichlet problem
\begin{equation}
\begin{cases}
-\Delta_g u_C=f & \text{in }B_g(p,\rho),\\
u_C=0 & \text{on }S_g(p,\rho).
\end{cases}
\label{eq:dirichlet-geodesic-ball}
\end{equation}

According to Theorem~\ref{thm:sufficient-condition}, a sufficient
condition for the existence of a minimizing domain
$\Omega^\ast$ satisfying
\[
C\subset\Omega^\ast
\]
and the free boundary condition
\[
|\nabla_g u_{\Omega^\ast}|_g^2
=
\sigma H_{\partial\Omega^\ast}+k^2
\qquad\text{on }\partial\Omega^\ast
\]
is
\begin{equation}
|\nabla_g u_C|_g^2
>
\sigma H_{\partial C}+k^2
\qquad\text{on }\partial C.
\label{eq:condition-geodesic-ball}
\end{equation}

In a general Riemannian manifold, the mean curvature
$H_{\partial C}$ depends on the geometry of $(M,g)$ and on the radius
$\rho$. Consequently, an explicit expression for
$H_{\partial C}$ is not available in general.

However, comparison geometry can be used to estimate the mean curvature
of geodesic spheres in terms of curvature bounds on $(M,g)$. Thus,
condition \eqref{eq:condition-geodesic-ball} can in principle be reduced
to explicit geometric estimates once suitable curvature assumptions are
imposed.

The following example provides a completely explicit situation.

\subsection{An explicit example on the sphere}
\label{subsec:sphere-example}

We now consider the $n$-dimensional sphere of radius $R$,
\[
S_R^n
=
\left\{
x\in\mathbb{R}^{n+1}:|x|=R
\right\},
\]
equipped with the Riemannian metric induced by the Euclidean metric.

Fix a point $p\in S_R^n$ and consider the geodesic ball
\[
C=B_g(p,\rho),
\qquad
0<\rho<\pi R.
\]

The geodesic distance from $p$ will be denoted by
\[
r(x)=d_g(p,x).
\]

In geodesic polar coordinates centered at $p$, the metric takes the form
\begin{equation}
g
=
dr^2
+
R^2\sin^2\left(\frac{r}{R}\right)
g_{\mathbb{S}^{n-1}},
\label{eq:sphere-polar-metric}
\end{equation}
where $g_{\mathbb{S}^{n-1}}$ denotes the standard metric on the unit
sphere $\mathbb{S}^{n-1}$.

The geodesic sphere
\[
\partial C=S_g(p,\rho)
\]
is therefore a hypersurface with induced metric
\begin{equation}
g_{\partial C}
=
R^2\sin^2\left(\frac{\rho}{R}\right)
g_{\mathbb{S}^{n-1}}.
\label{eq:induced-sphere-metric}
\end{equation}

With the outward unit normal
\[
\nu=\partial_r,
\]
the second fundamental form of $\partial C$ is
\begin{equation}
\mathrm{II}_{\partial C}
=
\frac{1}{R}
\cot\left(\frac{\rho}{R}\right)
g_{\partial C}.
\label{eq:second-fundamental-sphere}
\end{equation}

With our convention
\[
H_{\partial C}
=
\operatorname{div}_{\partial C}\nu,
\]
we consequently obtain
\begin{equation}
\boxed{
H_{\partial C}
=
\frac{n-1}{R}
\cot\left(\frac{\rho}{R}\right).
}
\label{eq:mean-curvature-sphere}
\end{equation}

Thus, the sufficient condition
\eqref{eq:condition-geodesic-ball} becomes
\begin{equation}
\boxed{
|\nabla_g u_C|_g^2
>
\frac{\sigma(n-1)}{R}
\cot\left(\frac{\rho}{R}\right)
+
k^2
\qquad\text{on }\partial C.
}
\label{eq:sufficient-sphere}
\end{equation}

This gives an explicit geometric form of the sufficient condition on
the sphere.

\subsection{Radially symmetric data on the sphere}
\label{subsec:radial-sphere}

We now assume that the right-hand side $f$ is radially symmetric with
respect to the point $p$. More precisely, we suppose that
\begin{equation}
f(x)=F(r(x)),
\label{eq:radial-f}
\end{equation}
where
\[
r(x)=d_g(p,x).
\]
In order to preserve the global assumption
\[
\operatorname{supp} f\subset K\subset C,
\]
we additionally assume that
\[
F(r)=0
\qquad\text{for all }r\ge r_0,
\]
for some fixed
\[
0<r_0<\rho,
\]
where $\rho$ denotes the radius of the geodesic ball under
consideration. Thus, the radial profile $F$ is compactly supported in
$[0,r_0)$, and consequently
\[
\operatorname{supp}f\subset \overline{B_g(p,r_0)}
\subset B_g(p,\rho)\subset C.
\]
We look for a radial solution of the form
\begin{equation}
u_C(x)=U(r(x)).
\label{eq:radial-u}
\end{equation}

For a radial function $U(r)$ on $S_R^n$, the Laplace--Beltrami operator
is given by
\begin{equation}
\Delta_g U
=
U''(r)
+
\frac{n-1}{R}
\cot\left(\frac{r}{R}\right)U'(r).
\label{eq:radial-laplacian-sphere}
\end{equation}

Therefore, the Dirichlet problem
\eqref{eq:dirichlet-geodesic-ball} reduces to the ordinary differential
equation
\begin{equation}
-
U''(r)
-
\frac{n-1}{R}
\cot\left(\frac{r}{R}\right)U'(r)
=
F(r),
\qquad
0<r<\rho,
\label{eq:radial-ode-sphere}
\end{equation}
with boundary condition
\begin{equation}
U(\rho)=0.
\label{eq:radial-boundary-sphere}
\end{equation}

Equation \eqref{eq:radial-ode-sphere} can be written in divergence
form as
\begin{equation}
-
\frac{1}
{\left(R\sin(r/R)\right)^{n-1}}
\frac{d}{dr}
\left[
\left(R\sin(r/R)\right)^{n-1}U'(r)
\right]
=
F(r).
\label{eq:radial-divergence-sphere}
\end{equation}

Equivalently,
\begin{equation}
\frac{d}{dr}
\left[
\left(R\sin(r/R)\right)^{n-1}U'(r)
\right]
=
-
\left(R\sin(r/R)\right)^{n-1}F(r).
\label{eq:radial-divergence-sphere-2}
\end{equation}

The regularity of $u_C$ at the center $p$ imposes
\begin{equation}
\lim_{r\to0}
\left(R\sin(r/R)\right)^{n-1}U'(r)
=
0.
\label{eq:regularity-center}
\end{equation}

Integrating \eqref{eq:radial-divergence-sphere-2} from $0$ to $r$ gives
\begin{equation}
\left(R\sin(r/R)\right)^{n-1}U'(r)
=
-
\int_0^r
\left(R\sin(s/R)\right)^{n-1}
F(s)\,ds.
\label{eq:radial-derivative-integral}
\end{equation}

Hence
\begin{equation}
\boxed{
U'(r)
=
-
\frac{
\displaystyle
\int_0^r
\left(R\sin(s/R)\right)^{n-1}
F(s)\,ds
}{
\left(R\sin(r/R)\right)^{n-1}
}.
}
\label{eq:radial-U-prime}
\end{equation}

Since
\[
\nabla_g u_C
=
U'(r)\partial_r
\]
and
\[
|\partial_r|_g=1,
\]
we have
\begin{equation}
|\nabla_g u_C|_g
=
|U'(r)|.
\label{eq:gradient-radial-sphere}
\end{equation}

In particular, on the boundary $r=\rho$,
\begin{equation}
\boxed{
|\nabla_g u_C|_g
=
\frac{
\displaystyle
\int_0^\rho
\left(R\sin(s/R)\right)^{n-1}
F(s)\,ds
}{
\left(R\sin(\rho/R)\right)^{n-1}
}.
}
\label{eq:boundary-gradient-sphere}
\end{equation}

Consequently,
\begin{equation}
|\nabla_g u_C|_g^2
=
\left[
\frac{
\displaystyle
\int_0^\rho
\left(R\sin(s/R)\right)^{n-1}
F(s)\,ds
}{
\left(R\sin(\rho/R)\right)^{n-1}
}
\right]^2.
\label{eq:boundary-gradient-square-sphere}
\end{equation}

Combining this identity with
\eqref{eq:sufficient-sphere}, we obtain the following explicit
criterion.

\begin{proposition}
\label{prop:explicit-sphere-condition}
Let $C=B_g(p,\rho)\subset S_R^n$, with $n\geq2$ and
$0<\rho<\pi R$. Assume that
\[
f(x)=F(d_g(p,x))
\]
and let $u_C$ be the solution of
\eqref{eq:dirichlet-geodesic-ball}. If
\begin{equation}
\boxed{
\left[
\frac{
\displaystyle
\int_0^\rho
\left(R\sin(s/R)\right)^{n-1}
F(s)\,ds
}{
\left(R\sin(\rho/R)\right)^{n-1}
}
\right]^2
>
\frac{\sigma(n-1)}{R}
\cot\left(\frac{\rho}{R}\right)
+
k^2,
}
\label{eq:explicit-condition-sphere}
\end{equation}
then the sufficient condition of
Theorem~\ref{thm:sufficient-condition} is satisfied.
Consequently, there exists a minimizing domain $\Omega^\ast$ such that
\[
C\subset\Omega^\ast
\]
and
\begin{equation}
|\nabla_g u_{\Omega^\ast}|_g^2
=
\sigma H_{\partial\Omega^\ast}+k^2
\qquad\text{on }\partial\Omega^\ast.
\label{eq:free-boundary-sphere-example}
\end{equation}
\end{proposition}

\begin{proof}
By \eqref{eq:boundary-gradient-square-sphere},
\[
|\nabla_g u_C|_g^2
=
\left[
\frac{
\displaystyle
\int_0^\rho
\left(R\sin(s/R)\right)^{n-1}
F(s)\,ds
}{
\left(R\sin(\rho/R)\right)^{n-1}
}
\right]^2
\]
on $\partial C$.

On the other hand, by
\eqref{eq:mean-curvature-sphere},
\[
H_{\partial C}
=
\frac{n-1}{R}
\cot\left(\frac{\rho}{R}\right).
\]

Therefore, condition
\eqref{eq:explicit-condition-sphere} is exactly
\[
|\nabla_g u_C|_g^2
>
\sigma H_{\partial C}+k^2
\qquad\text{on }\partial C.
\]

The conclusion follows directly from
Theorem~\ref{thm:sufficient-condition}.
\end{proof}

\subsection{A constant radial source}
\label{subsec:constant-source-sphere}

We choose the radial profile in the form
\[
F(r)=f_0\,\eta(r),
\]
where $f_0>0$ is a constant and
\[
\eta\in C^\infty([0,\rho]),
\qquad
0\leq\eta\leq1,
\]
is a smooth cut-off function satisfying
\[
\eta(r)=1
\quad\text{for }0\leq r\leq r_0,
\]
and
\[
\eta(r)=0
\quad\text{for }r\geq r_1,
\]
where
\[
0<r_0<r_1<\rho.
\]
Consequently,
\[
\operatorname{supp}F\subset[0,r_1],
\]
and, for the corresponding radial function
\[
f(x)=F\bigl(d_g(p,x)\bigr),
\]
we have
\[
\operatorname{supp}f
\subset\overline{B_g(p,r_1)}
\subset C.
\]
Thus there exists a compact set
\[
K:=\overline{B_g(p,r_1)}\subset C
\]
such that
\[
\operatorname{supp}f\subset K,
\]
Then \eqref{eq:radial-U-prime} becomes
\begin{equation}
U'(r)
=
-
f_0
\frac{
\displaystyle
\int_0^r
\left(R\sin(s/R)\right)^{n-1}\eta(s)\,ds
}{
\left(R\sin(r/R)\right)^{n-1}
}.
\label{eq:constant-source-Uprime}
\end{equation}

Consequently,
\begin{equation}
|\nabla_g u_C|_g^2
=
f_0^2
\left[
\frac{
\displaystyle
\int_0^\rho
\left(R\sin(s/R)\right)^{n-1}\eta(s)\,ds
}{
\left(R\sin(\rho/R)\right)^{n-1}
}
\right]^2
\qquad\text{on }\partial C.
\label{eq:constant-source-gradient}
\end{equation}

Hence a sufficient condition for the free boundary problem is
\begin{equation}
\boxed{
f_0^2
\left[
\frac{
\displaystyle
\int_0^\rho
\left(R\sin(s/R)\right)^{n-1}\eta(s)\,ds
}{
\left(R\sin(\rho/R)\right)^{n-1}
}
\right]^2
>
\frac{\sigma(n-1)}{R}
\cot\left(\frac{\rho}{R}\right)
+
k^2.
}
\label{eq:constant-source-condition}
\end{equation}


\subsection{The one-dimensional sphere $S_R^1$}
\label{subsec:S1-example}

For completeness, we briefly discuss the one-dimensional case.

Let
\[
S_R^1
=
\left\{
(R\cos\theta,R\sin\theta):
\theta\in[0,2\pi)
\right\}.
\]

The induced metric is
\begin{equation}
g=R^2\,d\theta^2,
\label{eq:S1-metric}
\end{equation}
and the Riemannian volume element is
\begin{equation}
dV_g=R\,d\theta.
\label{eq:S1-volume}
\end{equation}

For a smooth function $u=u(\theta)$, we have
\begin{equation}
\nabla_g u
=
\frac{1}{R^2}
\frac{du}{d\theta}
\frac{\partial}{\partial\theta},
\label{eq:S1-gradient}
\end{equation}
and therefore
\begin{equation}
|\nabla_g u|_g^2
=
\frac{1}{R^2}
\left|\frac{du}{d\theta}\right|^2.
\label{eq:S1-gradient-norm}
\end{equation}

The Laplace--Beltrami operator is
\begin{equation}
\Delta_g u
=
\frac{1}{R^2}
\frac{d^2u}{d\theta^2}.
\label{eq:S1-laplacian}
\end{equation}

Thus, on an interval of $S_R^1$, the Dirichlet problem
\[
-\Delta_g u=f
\]
takes the form
\begin{equation}
-\frac{1}{R^2}u''(\theta)
=
f(\theta).
\label{eq:S1-dirichlet}
\end{equation}

The case $S_R^1$ is useful for checking explicitly the scaling of the
Riemannian gradient and Laplace--Beltrami operator with respect to the
radius $R$. However, since the boundary of a one-dimensional domain is
zero-dimensional, the mean-curvature formulation used in the free
boundary condition is naturally considered for $n\geq2$.

\subsection{Conclusion}
\label{subsec:examples-conclusion}

The preceding examples show that the sufficient condition
\[
|\nabla_g u_C|_g^2
>
\sigma H_{\partial C}+k^2
\]
can be made explicit in geometrically relevant Riemannian settings.

In particular, for a geodesic ball on the sphere $S_R^n$ and a radial
source term, the condition reduces to the explicit inequality
\[
\left[
\frac{
\displaystyle
\int_0^\rho
\left(R\sin(s/R)\right)^{n-1}
F(s)\,ds
}{
\left(R\sin(\rho/R)\right)^{n-1}
}
\right]^2
>
\frac{\sigma(n-1)}{R}
\cot\left(\frac{\rho}{R}\right)
+
k^2.
\]

This provides a concrete class of examples for which the abstract
existence criterion of Theorem~\ref{thm:sufficient-condition} can be
verified directly.


\begin{thebibliography}{99} 

\bibitem{BLS}
Barkatou, M., Seck, D., and Ly, I. (2005). An existence result for a quadrature surface free boundary problem. Open Mathematics, 3(1):39--57.
\bibitem{BLS2}
Barkatou, Mohammed and Seck, Diaraf and Ly, Idrissa. (2006). An existence result for an interior electromagnetic casting problem. Central European Journal of Mathematics, 4(4):573--584.


\bibitem{Da1}
Dambrine, M. (2002). On variations of the shape hessian and sufficient conditions for the stability of critical shapes. Racsam, 96:95--121.
\bibitem{DZ}
Delfour, M. C. and Zol{\'e}sio, J.-P. (2011). Shapes and geometries: metrics, analysis, differential calculus,
  and optimization. SIAM.
\bibitem{DS1}
Djit{\'e}, A. S. and Seck, D. (2022). A riemannian point of view for a quadrature surface free boundary
  problem. Nonlinear Analysis, Geometry and Applications: Proceedings of
  the Second NLAGA-BIRS Symposium, Cap Skirring, Senegal, January 25--30,
  2022, pages 339--374. Springer.
\bibitem{He}
Henrot, A. (1994). Subsolutions and supersolutions in a free boundary problem.
Arkiv f{\"o}r Matematik, 32(1):79-98.

\bibitem{HP}
Henrot, A. and Pierre, M. (2018). Shape variation and optimization. European Mathematical Society.

\bibitem{Schu}
Schulz, V. H. (2014). A riemannian view on shape optimization. Foundations of Computational Mathematics, 14:483--501.

\bibitem{sec}
Seck, D. (2016). On an isoperimetric inequality and various methods for the bernoulli’s free boundary problems. S{\~a}o Paulo Journal of Mathematical Sciences, 10:36--59.

\bibitem{sozo}
 Sokolowski, J., and Zolesio, J.-P. (1992).
Introduction to shape optimization. Springer, Berlin, Heidelberg.

\end{thebibliography}
\end{document}